\RequirePackage{fix-cm}
\documentclass{amsart}    
\usepackage{bbm}
\usepackage{graphicx}
\usepackage{amsmath}
\usepackage{amssymb}
\usepackage{wasysym}
\usepackage{mathrsfs}
\usepackage{geometry} 
\usepackage{mathtools}

\DeclareMathOperator{\R}{\mathbb{R}}

\DeclareMathOperator{\as}{as}

\DeclareMathOperator{\conv}{conv}

\DeclareMathOperator{\dell}{del}
\DeclareMathOperator{\divv}{div}
\DeclareMathOperator{\diam}{diam}
\DeclareMathOperator{\var}{Var}
\DeclareMathOperator{\intt}{int}

\let\originalleft\left
\let\originalright\right
\renewcommand{\left}{\mathopen{}\mathclose\bgroup\originalleft}
\renewcommand{\right}{\aftergroup\egroup\originalright}

\newtheorem{theorem}{Theorem}[section]
 
\newtheorem{corollary}[theorem]{Corollary}
\newtheorem{lemma}[theorem]{Lemma}
\newtheorem{remark}[theorem]{Remark}
\newtheorem{conj}[theorem]{Conjecture}
\newtheorem{ques}[theorem]{Question}
\newcommand{\NN}{N^{-\frac 2{n-1}}}

\newcommand{\Sn}{\mathbb{S}^{n-1}}
\newcommand\inneri[2]{\langle #1, #2 \rangle}
\newcommand{\Sp}{\mathbb{S}^{n-1}}

\newcommand{\E}{{\mathbb{E}}}

\begin{document}

\title[Random Polytopes, Dirichlet-Voronoi Tiling Numbers and Geometric Balls and Bins  ]{A Concentration Inequality for Random Polytopes, Dirichlet-Voronoi Tiling Numbers and the Geometric Balls and Bins Problem}



\author{Steven Hoehner}     
\address{Longwood University, Department of Mathematics and Computer Science}
\email{hoehnersd@longwood.edu}
\author{Gil Kur} 
\address{Weizmann Institute of Science,
               Department of Computer Science and Mathematics}
               \email{gilkur1990@gmail.com}


\subjclass[2000]{Primary 52A22, 52A27, 52C17, 52C22}
\maketitle

\begin{abstract}
	Our main contribution is a concentration inequality for the symmetric volume difference of a $ C^2 $ convex body with  positive Gaussian curvature and a circumscribed random polytope with a restricted number of facets, for any probability measure on the boundary with a  positive density function.
	 	 
We also show that the Dirichlet-Voronoi tiling numbers satisfy $ \text{div}_{n-1} = (2\pi e)^{-1}(n+\ln n) + O(1)$, which improves a classical result of Zador by a factor of $o(n)$. In addition, we provide a remarkable open problem which is the natural geometric generalization of the famous and fundamental ``balls and bins" problem from probability.   This problem  is tightly connected to the optimality of random polytopes in high dimensions.

Finally, as an application of the aforementioned results, we derive a lower bound for the maximal Mahler volume product of polytopes with a restricted number of vertices or facets.
\keywords{ Random polytopes \and approximation\and convex bodies \and sphere covering \and tiling}
\end{abstract}


\section{Introduction and main results}

The approximation of convex bodies by polytopes is of significant theoretical interest in convex geometry, and it has applications in a wide variety of areas, including tomography (e.g., \cite{gardner1995geometric}), computational geometry (e.g., \cite{edelsbrunner,edelsbrunnermesh}), geometric algorithms (e.g., \cite{GKM}) and statistical learning theory (e.g., \cite{diakonikolas2016learning}). The accuracy of the approximation  is often measured by the symmetric volume difference $d_S$ (also called the symmetric difference metric, or the Nikodym metric), which equals the volume of the symmetric difference of a given convex body $K$ in $\R^n$ and an approximating polytope $P$. It is defined by
$d_S(K,P) := |K\triangle P|,$ where $|\cdot|$ denotes $n$-dimensional volume. 

Typically, conditions are prescribed on the approximating polytopes, such as a restricted number of vertices or facets. Moreover, the approximating polytopes may be inscribed in $K$, circumscribed around $K$, or positioned arbitrarily. In this paper, we focus on the approximation of convex bodies by circumscribed and arbitrarily positioned polytopes with a restricted number of facets under the symmetric volume difference. 

Due to the difficulty of finding  best-approximating polytopes (even when $n = 3$), random polytopes can and have been used to derive sharp estimates for the approximation of convex bodies. In fact, it turns out that, asymptotically, random polytopes are almost as good as best-approximating polytopes under the symmetric difference metric \cite{boroczky2004approximation,GruberII,kur2017approximation,ludwig1999asymptotic,Lud06,muller1990approximation,schutt2003polytopes}. 
In our first theorem, we derive a sharp concentration inequality (up to logarithmic factors)  for the volume difference of a $ C^2 $ convex body $K$ with  positive Gaussian curvature and a random circumscribed polytope with $N$ facets. 
The expected volume of these polytopes was estimated in \cite{boroczky2004approximation}. 



\begin{theorem}{\label{geneThm}}
	Let $ P_{n,N} $ be a random circumscribed polytope in $ \R^{n} $ with at most $N$ facets that is defined by
	\begin{equation}\label{Eq:def}
	P_{n,N} := \bigcap_{i=1}^{N}\{x\in\R^n: \inneri{x}{\nu(X_i)} \leq \inneri{X_i}{\nu(X_i)} \}, \quad X_1,\ldots,X_N \stackrel{\text{i.i.d.}}{\sim} \mu,
	\end{equation} where   $ \nu:\partial K \to \Sp $ denotes the Gauss map of $K$. 
	Then when $ N$  is large enough, for any $ \epsilon > 0 $ the following holds:
	\begin{equation}{\label{maineq}}
	\Pr_A(\big||P_{n,N} \setminus K| - \E[|P_{n,N}\setminus K|]\big| \geq \epsilon) \leq 2\exp\left(-c(K,\mu)N^{1+\frac{4}{n-1}}f(N)\epsilon^2\right),
	\end{equation}
	where $A$ is an event that holds with probability at least $1- \exp\left(-c_1(K,\mu)N^{0.5-\frac{2}{n-1}}\ln N\right)$, $c(K,\mu)$ and $c_1(K,\mu)$ are  positive constants that depend on $K$ and $\mu$, and $ f(N) = (\ln N)^{-(2+\frac{4}{n-1})}$. 
\end{theorem}
This result implies that the conditional variance is bounded above by $C(K,\mu)N^{-(1+\frac{4}{n-1})}f(N)^{-1}$.

The proof of Theorem \ref{geneThm} also gives a concentration inequality for the volume of the arbitrarily positioned random polytopes that were defined in \cite{kur2017approximation}; these polytopes give an optimal approximation to the Euclidean unit ball $B_n$ when $\mu=\sigma$ is the uniform measure on the sphere.
\begin{corollary}{\label{mainThm}}
	Let $ P_{n,N} $ be the random polytope in $ \R^{n} $ that is defined by
	\[
	P_{n,N} := \bigcap_{i=1}^{N}\{x\in\R^n: \inneri{X_i}{x} \leq t_{n,N} \}, 
	\quad X_1,\ldots,X_N \stackrel{\text{i.i.d.}}{\sim}  \sigma,
	\] 
	where $t_{n,N} := {\sqrt{1-\varepsilon_{n,N}^2}}$ and $\varepsilon_{n,N} := \left(\frac{(\ln 2)|\partial B_n|}{N|B_{n-1}|}\right)^{\frac{1}{n-1}}$.	 
	Then when $N$ is large enough, the random variable $|P_{n,N} \triangle B_n|$ satisfies the concentration inequality
	\[
	\Pr_A(\big||P_{n,N}\triangle B_n|-\E[|P_{n,N}\triangle B_n|]\big| \geq \epsilon) 
	\leq 2\exp\left(-c(n)N^{1+\frac{4}{n-1}}f(N)\epsilon^2\right),
	\]
	where $A$ is an event that holds with probability at least $\exp\left(-c_1(n)N^{0.5-\frac{2}{n-1}}\ln N\right)$, $c(n)$ and $c_1(n)$ are positive constants that depend only on the dimension, and $ f(N)=(\ln N)^{-(2+\frac{4}{n-1})}$. 
\end{corollary}

Using 
 an argument of Reitzner \cite{reitzner2003random}, we derive the following corollary to Theorem \ref{geneThm}.
\begin{corollary}\label{corollary10}
	Consider the sequence of circumscribed random polytopes $(P_{n,N})_{N}$ defined  in Theorem \ref{geneThm}. Then with probability 1,
	\[
	\lim_{N\to\infty} N^{\frac{2}{n-1}}|P_{n,N}\setminus  K| = \frac{1}{2}|B_{n-1}|^{-\frac{2}{n-1}} \Gamma\left(\frac{2}{n-1}+1\right)\int_{\partial K}g(x)^{-\frac{2}{n-1}} \kappa(x)^{\frac{1}{n-1}} \,dx,
	\] 
	where $g(x)\,dx = d\mu(x)$ and $g(x)>0$.
\end{corollary}

Our next result is an application of \cite{kur2017approximation} and provides an improvement to a result of Zador \cite{zador1982asymptotic} on the asymptotic behavior of the Dirichlet-Voronoi tiling number in $\R^n$, denoted by $\divv_{n-1}$.
These numbers have numerous definitions. From a geometric point of view, Gruber \cite{GruberII} proved that for any convex body $ K $ in $\R^n$ with  $C^2$ boundary and  positive Gaussian curvature $\kappa$, the following asymptotic formulas hold:
\begin{align}\label{GruberOutEqn}
\lim_{N\to\infty}N^{\frac{2}{n-1}}\min\{|P\setminus K|: P\supset K, P\text{ has at most }N \text{ facets}\} &= \frac{1}{2}\divv_{n-1}(\as(K))^{\frac{n+1}{n-1}}\\
\label{GruberInEqn}
\lim_{N\to\infty}N^{\frac{2}{n-1}}\min\{|K\setminus P|: P\subset K, P\text{ has at most }N \text{ vertices}\} &= \frac{1}{2}\dell_{n-1}(\as(K))^{\frac{n+1}{n-1}}.
\end{align}
Here $\as(K)=\int_{\partial K}\kappa(x)^{\frac{1}{n+1}}\,dS(x)$ is the affine surface area of $K$ (see, e.g., \cite{SchneiderBook}). 
The term $\dell_{n-1}$ is a positive constant that depends only on the dimension and is known as the Delone triangulation number in $\R^n$. A strong connection between  optimal Delone triangulations, sphere covering, and asymptotic best approximation of convex bodies by inscribed polytopes with a restricted number of vertices was exhibited by Chen \cite{chen}.

The exact values of $\dell_{n-1}$ and $\divv_{n-1}$ are unknown for $ n \geq 4 $. For $ n=2$, it is known\footnote{Please note that for $n=2$, Eqs. (\ref{GruberOutEqn}) and (\ref{GruberInEqn})  were stated by T\'oth \cite{toth} and proved by McClure and Vitale \cite{McVi}; for $n=3$, they were also stated in \cite{toth}, and later proved by Gruber in \cite{GruberIn,GruberOut}.} that $\dell_1=\frac{1}{6}$ and $\divv_{1}=\frac{1}{12}$ (see, e.g., \cite{ludwig1999asymptotic}), and for $n=3$ the values  $\dell_2=\frac{1}{2\sqrt{3}}$ and $\divv_2=\frac{5}{18\sqrt{3}}$ were determined by Gruber in \cite{GruberIn} and \cite{GruberOut}, respectively. Surprisingly, for large $n$, the constants $\dell_{n-1}$ and $\divv_{n-1}$ are nearly equal. First,  in the groundbreaking paper \cite{zador1982asymptotic}, Zador proved that $\divv_{n-1} =(2\pi e)^{-1}n +o(n)$.  On the other hand, $\dell_{n-1}$ was estimated in several papers. The best known estimate was provided by Mankiewicz and Sch\"utt \cite{MaS1}, who proved that 
$\dell_{n-1} = (2\pi e)^{-1}n + O(\ln n)$. In the next theorem, we provide an almost sharp estimate for $\divv_{n-1}$. 

\begin{theorem}\label{deldiv}
	Let  $\divv_{n-1} $  be defined as in Eq. \eqref{GruberOutEqn}. Then 
	\begin{equation}\label{deldiv2}
	\divv_{n-1} = (2\pi e)^{-1}(n+\ln n)+O(1).
	\end{equation}
\end{theorem}

\noindent Theorem \ref{deldiv} and the estimate for $\dell_{n-1}$ imply that, asymptotically, $ \dell_{n-1} $ and $ \divv_{n-1} $ differ by at most $O(\ln n)$.

\vskip 2mm
 
%

 
%


 The following result is a remarkable application of Theorem \ref{deldiv} and the main result in  \cite{boroczky2004approximation} (the asymptotic notation $O(\cdot)$  is with respect to the dimension $n$).
\begin{corollary}\label{randomvsbest}
	Let $ P_{K,N}^b $ be a best-approximating polytope with at most $N$ facets that circumscribes $K$, and let $ P_{n,N} $ be a random polytope that is generated as in Theorem \ref{geneThm} with the density that minimizes the expectation of the volume difference, which is  $ (\as(K))^{-1}\kappa^{\frac{1}{n+1}}\mathbbm{1}_{\partial K}$. Then
		\begin{equation}\label{noCovexlim}
	\lim_{N\to\infty}\frac{\E[|P_{n,N}\setminus K|]}{|P_{K,N}^b \setminus K|} 
	= 1+O(\tfrac 1n).
	\end{equation}
\end{corollary}

Our last result is a lower bound for the  functional
\[
F(n,N) := \max_{P_{n,N} \subset \R^{n} \text {has at most $N$ vertices}} |P_{n,N}|\cdot|P^{s(P_{n,N})}_{n,N}|,
\] where $s(P_{n,N})  $ denotes the Santal\'o point of the polytope $ P_{n,N} $ and  $ P^{s(P_{n,N})}_{n,N} $ is the polar body of $P_{n,N}$ with respect to $s(P_{n,N})$ (when the Santal\'o point is the origin, we use the notation of $K^\circ$ to denote the polar body of $K$). Recently, Alexander, Fradelizi and Zvavitch \cite[Theorem 3.4]{alexander2017polytopes} proved that the maximum of $F(n,N)$  is achieved by a simplicial polytope with precisely $N$ vertices. We use all of the results in this paper to prove the following theorem. 
\begin{theorem}\label{mahler}
	Let $ P_{n,N} $ be the centrally symmetric random polytope
	\[
	P_{n,N} := \conv\{\pm X_i\}_{i=1}^{\frac{N}{2}} \subset B_n, \quad X_1,\ldots,X_{\frac{N}{2}} \stackrel{\text{i.i.d.}}{\sim} \sigma(\Sn).
	\]
	Then there exists an absolute constant $C>0$ such that when $ N \geq C^nn^{0.5n}$, with overwhelming probability it holds that 
	\[
	F(n,N)  \geq |P_{n,N}|\cdot|P_{n,N}^{\circ}| \geq \E[|P_{n,N}|\cdot|P_{n,N}^{\circ}|] + c\NN|B_n|^2 \geq (1 - (1.5\ln n + O(1))\NN)|B_n|^2.
	\]
%
\end{theorem}

We conclude this section with a few remarks.

\begin{remark}
	For convex bodies $K$ and $L$ in $\R^n$, the surface area deviation $\triangle_s(K,L)$ is defined by (see, e.g., \cite{Wer15,kur2017approximation})
	\[
	\triangle_s(K,L) := |\partial (K\cup L)| - |\partial (K\cap L)|.
	\]
	Kur \cite{kur2017approximation} showed that the polytopes $P_{n,N}$ of Corollary \ref{mainThm} satisfy $\triangle_s(P_{n,N},B_n) \leq 4n|P_{n,N}\triangle B_n|$. 
This inequality and Corollary \ref{mainThm} together imply that for all sufficiently large $N$, the following ``large deviation" inequality holds:
	\[
	\Pr\left( \triangle_s(B_n,P_{n,N}) > C \cdot \E\left[\triangle_s(B_n,P_{n,N})\right] \right) \leq 2 \exp\left(-c(n)N^{1-\frac{2}{n-1}}\right).
	\]
\end{remark}
\begin{remark}
	In the case of inscribed random polytopes whose vertices are chosen randomly from the boundary of the body $K$, one can use the parallel results of \cite{MaS2,schutt2003polytopes} (with density $\kappa^{\frac{1}{n+1}}\mathbbm{1}_{\partial K}$) to deduce a limit formula like \eqref{noCovexlim} with a convergence rate that is bounded by $ 1+O(n^{-1}\ln n)$ as $N\to\infty$.
\end{remark}
\section{Prior work}

\subsection{Random polytopes} 

There is a rich literature on the statistical properties of inscribed random polytopes with a restricted number of vertices. It would be impossible to list all of the results in this direction, so we will highlight those  which are most relevant to this paper. First, M\"uller \cite{muller1990approximation} showed that when the vertices are chosen uniformly and independently from the boundary of the Euclidean ball, the expectation of the volume difference is optimal up to an absolute constant. Sch\"utt and Werner \cite{schutt2003polytopes} generalized this result to any $ C^2 $ convex body with 
positive Gaussian curvature, and for any continuous, positive density on the boundary of the body. They also derived an explicit formula for the optimal density function that minimizes the expected volume difference over all choices of positive densities, and showed that if the optimal density is chosen, then as the dimension tends to infinity, random approximation is asymptotically as good as best approximation.

From results on the expectation, an immediate question that follows is to investigate higher moments of the volume difference. K\"ufer \cite{kufer1994approximation} proved an inequality for the variance of the volume of a random polytope that is the convex hull of points chosen uniformly and independently from the boundary of the Euclidean ball. Reitzner \cite{reitzner2003random} later extended this result to all $C^2$ convex bodies $K$ with  positive generalized Gaussian curvature by showing that the variance of the volume of a random polytope whose vertices are chosen uniformly and independently from $\partial K$ is at most  $C(K)N^{-(1+\frac{4}{n-1})}$, where $C(K)$ is a positive constant that depends on $K$.  In \cite{reitzner2003random} it was also shown that the variance of the volume of a random polytope whose vertices are chosen uniformly and independently from $ K$  is at most $C(K)N^{-(1+\frac{2}{n+1})}$. Later,  Reitzner \cite{reitzner2005central} also proved a matching  lower bound for the variance of the same order.


Another question that arises is to investigate concentration of the volume  of random polytopes. Vu \cite{vu2005sharp} used a ``boosted" martingale method to prove a sharp concentration inequality for the symmetric volume difference of a smooth convex body $K$ with $C^2$ boundary and $|K|=1$, and a random inscribed polytope $P_{n,N}$ that is the convex hull of $N$ points chosen uniformly and independently from $ K.$  
More specifically, it was shown that there exist positive constants $c_1(K)$ and $c_2(K)$ depending only on $K$ such that for any $\lambda\in\left(0,\frac{1}{4}c_1(K)N^{-\frac{(n-1)(n+3)}{(n+1)(3n+5)}}\right]$,
\begin{align}
\Pr\bigg(\big||K\setminus P_{n,N}|-\E[|K\setminus P_{n,N}|]\big| &\geq \sqrt{c_1(K)\lambda N^{-\left(1+\frac{2}{n+1}\right)}}\bigg)  \nonumber \\
&\leq 2\exp(-\lambda/4)+\exp\left(-c_2(K) N^{\frac{n-1}{3n+5}}\right).{\label{bbb}}
\end{align}Later, Vu \cite{vu2006central}  used a central limit theorem of Reitzner \cite{reitzner2005central} for Poisson point processes and the concentration inequality \eqref{bbb} to prove a central limit theorem for random polytopes inscribed in a $C^2$ convex body with positive curvature.

Although there are numerous results on the statistical properties of random polytopes with a restricted number of vertices,  much less is known about the case of random polytopes with a restricted number of facets. In this direction, B\"or\"oczky and Reitzner \cite{boroczky2004approximation} calculated the expectation of the volume difference of a smooth convex body $K$ and a random circumscribed polytope with $N$ facets. The random polytope was generated as follows: Choose $ N$ i.i.d. random points from the boundary of $K$  with respect to a given density function, and take the intersection of the supporting hyperplanes at these points. In \cite{boroczky2004approximation}, the optimal density that minimizes the expected volume difference was also determined explicitly in terms of $K$. 

Recently, Fodor, Hug, and Ziebarth \cite{fodor2016volume} computed the expectation and the variance of the volume difference of $K$ and the polar  $ P_{n,N}^{\circ}$ of the random polytope from (\ref{bbb})  for a general density on $K$. 
Surprisingly, the estimates for the variance in \cite{reitzner2003random} and  \cite{fodor2016volume} are equal, up to a constant that depends on $ K.$ 

To the best of the authors' knowledge, this paper is the first  to show concentration results for the volume of random polytopes generated by a probability measure which is not necessarily  uniform, and it is also the first to show concentration results for the volume of random polytopes with a restricted number of facets. We believe that this work is a natural extension of \cite{boroczky2004approximation}.

\subsection{The Mahler volume product of polytopes}

Let $K$ be a convex body in $\R^n$. The Santal\'o point $s(K)$ of $K$ is the unique point that satisfies $|K^{s(K)}|=\min_{z\in\intt(K)}|K^z|$, where  $K^z:=(K-z)^\circ$ is the polar body of $K$ with center of polarity $z$ and $\intt(K)$ denotes the interior of $K$. 
The Mahler volume product  of $K$ is then defined by $|K|\cdot|K^{s(K)}|$. The maximizers of the volume product are precisely the ellipsoids; this celebrated result is called the Blaschke-Santal\'o inequality. It was proved for $n\leq 3$ by Blaschke \cite{Blaschke1917,Blaschke1918} and extended to all dimensions by Santal\'o \cite{santalo1949}. The equality conditions were proved by Saint-Raymond \cite{saintraymond1981} in the symmetric case, and by Petty \cite{petty1985} for the general case.  

On the other hand, finding the minimizers of the  volume product is a major open problem in convex geometry, which has attracted considerable interest. {\it Mahler's conjecture} \cite{mahler1} states  that the simplex is the minimizer. It was proven for $n=2$ by Mahler \cite{mahler2}, and Meyer \cite{meyer1991} proved that equality is attained only for triangles. The conjecture remains open for $n\geq 3$. In the case the body $K$ is centrally symmetric, Mahler \cite{mahler1} conjectured that the minimizer is the unit cube, and he proved it in \cite{mahler2} for $n=2$. For $n=3$, the symmetric case was proved in the affirmative by Iriyeh and Shibata \cite{iriyehshibata} (recently, the proof was simplified in \cite{equipartition}). In general dimensions, it was shown by Nazarov et. al. \cite{NPRZ} that the unit cube is a strict local minimizer for the volume product among symmetric bodies. Nevertheless, Mahler's conjecture for symmetric bodies remains open for $n\geq 4$.

Recently, the volume product of polytopes with a fixed number of vertices was considered by Alexander, Fradelizi and Zvavitch \cite{alexander2017polytopes}. They showed that the functional
\[
F(n,N):=\max_{P_{n,N} \subset \R^{n} \text {has at most $N$ vertices}} |P_{n,N}|\cdot|P^{s(P_{n,N})}_{n,N}|,
\]
achieves its maximum for simplicial polytopes. 
In Theorem \ref{mahler}, we use all of the results in this paper, as well as  results of M\"uller \cite{muller1990approximation} and Reitzner \cite{reitzner2003random}, to derive a lower bound for the asymptotic behavior of $ F(n,N)$.  For more background, as  please see Section \ref{resultstable}. The proof of Theorem \ref{mahler} is in Section \ref{mahlerproof1}.

\section{Discussion: Optimality of random polytopes in high dimensions}\label{resultstable}

The following table summarizes known results on the asymptotic behavior of the quantity $$\frac{\E[d(B_n,P_{n,N})]}{ d(B_n,P_{n,N}^b)},$$ 
where $P_{n,N}^b$ is the polytope with $N$ facets or $N$ vertices that is best-approximating with respect to a metric $d$ under a constraint that appears in the first column of the table. The distance $d$ is either the volume difference or Hausdorff metric, and the expectation is taken with respect to the uniform distribution (with the assumption that all the facets or vertices have the same height). In the inscribed and circumscribed cases the height is one, whereas in the arbitrary case the height is chosen to minimize the distance (in expectation).

The values in the table are given for the following types of polytopes: circumscribed with $N$ facets, inscribed with $N$ vertices  and arbitrarily positioned with either $N$ facets or $N$ vertices.
\begin{table}[h]
\centering
	\begin{tabular}{|l|l|l|}
		\hline
		& \multicolumn{1}{|c|}{Volume difference}                                      & \multicolumn{1}{|c|}{Hausdorff metric}                          \\ \hline
		Facets circumscribed           & $1+O(n^{-1})$   \,\,\,\, \cite{boroczky2004approximation,GruberII}, Thm. \ref{randomvsbest}         & $\Omega(1)\ln N$ \,\,\,\,\cite{boroczky2000polytopal,glasauer1996asymptotic,janson1986random} 
		\\ \hline
		Vertices inscribed      & $1+O(n^{-1}\ln n)$ \,\,\,\,\cite{GruberII,MaS1,MaS2,schutt2003polytopes} &            $\Omega(1)\ln N$ \,\,\,\,\cite{boroczky2000polytopal,glasauer1996asymptotic,janson1986random}                \\ \hline
		Arbitrary facets & $\Theta(1)$     \,\,\,\,\cite{kur2017approximation,Lud06}   &  $\Omega(1)\ln N$ \,\,\,\,\cite{boroczky2000polytopal,glasauer1996asymptotic}\\
		\hline
		Arbitrary vertices & $O(1)$  \,\, \cite{Lud06}, \,\, $\Omega(n^{-1})$ \,\, \cite{boroczky2000polytopal}  &  $\Omega(1)\ln N$ \,\,\,\,\cite{boroczky2000polytopal,glasauer1996asymptotic}\\
		\hline
	\end{tabular}
\end{table}

Our main question on the optimality of random polytopes in high dimensions is in the spirit of Corollary \ref{randomvsbest}. In order to formulate it,  we use the main result of Kur \cite{kur2017approximation}, where he found the arbitrarily positioned random polytopes with $N$ facets that minimize the expectation of the symmetric volume difference with the Euclidean ball. Moreover, these random polytopes are optimal up to absolute constants. Namely, by \cite[Theorem 2]{Lud06} and \cite[Remark 2.2]{kur2017approximation}, for all $N \geq 10^n$ it holds that
\[
	c\NN|B_n| \leq |P^b_{n,N} \triangle B_n| \leq \E[|P_{n,N} \triangle B_n|] \leq C\NN|B_n|,
\]
where $P^b_{n,N}$ is the best-approximating polytope with $N$ facets under the constraint that its facets have roughly the same height\footnote{We mean that for every $n \geq 3, N > 10^{n}$ there exists a polytope $P_{n,N}^h$ in $\R^n$ with $N$ facets that all have (roughly) the same height $h$, which satisfies $|P^b_{n,N}|-|P^h_{n,N}| = o(1)\NN$.} (see Corollary~\ref{mainThm}). Thus, it is natural to ask if the following formula holds:
\begin{equation}{\label{Eq:Conj}}
	\lim_{N\to\infty}\frac{\E[|P_{n,N} \triangle B_n|]}{|P^b_{n,N} \triangle B_n|} = 1+o(1).
\end{equation}
\noindent{In other words, Eq. \eqref{Eq:Conj} asks if random polytopes become ``more optimal" as the dimension increases.}

We strongly believe that Eq. \eqref{Eq:Conj} holds. For example, these random polytopes satisfy the property that as $N$ increases, half of their surface area is outside, which is a condition that best-approximating polytopes must satisfy (as was observed in \cite[Lemma 9]{Lud06}).

Let $X_1,\ldots,X_N\stackrel{\text{i.i.d.}}{\sim} \sigma$ and $r_{\gamma} :=  {\sqrt{1-\left(\frac{\gamma|\partial B_n|}{N|B_{n-1}|}\right)^{\frac{2}{n-1}}}}$ for any fixed $\gamma \in (0,\infty).$ By Eqs. (\ref{Eq:refine}) and (\ref{Eq:refine2}) below, the quantity 
\begin{equation}\label{quantity1}
     \left|\E\left[\int_{\mathbb{S}^{n-1}}\mathbbm{1}_{\{\max\langle X_{i},x\rangle\leq r_\gamma\}}(x)\,d\sigma(x)\right] - \max_{x_1,\ldots,x_N}\int_{\mathbb{S}^{n-1}}\mathbbm{1}_{\{\max\langle x_{i},x\rangle\leq r_\gamma\}}(x)\,d\sigma(x)\right|
\end{equation}
plays a key role in understanding the question \eqref{Eq:Conj}. This quantity also has a beautiful interpretation that is connected to the ``geometric balls and bins problem", which we define in the next section. Finally, we conjecture that if for every fixed $\gamma \in (0,\infty)$ the quantity \eqref{quantity1} is of the order $o(1)$, then Eq. \eqref{Eq:Conj} holds.


\section{Geometric generalization of the ``balls and bins" problem}{\label{sphercovering}}

It turns out that the optimality of random polytopes is closely connected to the optimality of random partial sphere coverings. As we shall see, the latter problem is the geometric generalization of the famous ``balls and bins" problem from probability, which we now recall.

Assume that we have $\lceil \alpha N \rceil$ balls ($\alpha \in (0,\infty)$) and $N$ bins. For each ball, we draw a bin uniformly and independently, and we place the ball inside of this bin. 
\begin{enumerate}
	\item By the linearity of expectation, a proportion of $1-e^{-\alpha} + 0.5\alpha e^{-\alpha}N^{-1} + O( c(\alpha)N^{-2})$ of the bins are occupied.
	\item Trivially, when $\alpha=1$ the best configuration of balls places one ball in one bin, i.e., all the bins are occupied.
	\item  If we have $CN\ln N$ balls that are drawn uniformly and independently, where $C>1$ is an absolute constant, then with high probability all of the bins will be full (provided $C$ is large enough).
\end{enumerate}

Next, we can think about the natural and equivalent ``continuous" question. Namely, we replace the $N$ bins by the interval $[0,1]/\{0 \sim 1\} =: \mathbb{S}^{1}$, and the $\lceil \alpha N \rceil$ balls by $\lceil \alpha N \rceil$ subintervals of length $N^{-1}$. Then, we draw a point uniformly and independently from $\mathbb{S}^{1}$ and  place a subinterval with length $N^{-1}$ centered at the drawn point. This ``continuous" version of the balls and bins problem has the same properties as the original: 
\begin{enumerate}
	\item By Fubini's Theorem, the expectation of the length of the union of the subintervals equals $1-e^{-\alpha} + 0.5\alpha e^{-\alpha}N^{-1} + O( c(\alpha)N^{-2})$ (see Eq. \eqref{indepRV} below).
	\item Trivially, when $\alpha=1$ the best configuration places $N$ subintervals end-to-end, and the length of their union is $1$.
	\item By \cite[Theorem 1.2]{janson1986random}, if we have $CN\ln N$ disjoint intervals with length $N^{-1}$, then with high probability their union will be the entire sphere $\mathbb{S}^{1}$.
\end{enumerate}

The  balls and bins problem for $\Bbb{S}^1$ can be extended naturally to the sphere $\Sn$ of  arbitrary dimension $n$ by replacing the subintervals of length $N^{-1}$ with geodesic balls of normalized surface measure $N^{-1}$. In this setting, the geometry of the sphere starts to have an effect. 
For example, (up to a set of measure zero) we cannot have $N$ disjoint geodesic balls of volume $N^{-1}$ when $n \geq 3$. Therefore, some of the aforementioned properties may change.
\begin{enumerate}
	\item Remarkably, the expectation of the union of the geodesic balls is dimension-free, and it equals $1-e^{-\alpha} + 0.5\alpha e^{-\alpha }N^{-1} + O( c(\alpha) N^{-2})$ in all dimensions  (see Eq. \eqref{indepRV} below).
	\item  Erd{\H{o}}s, Few and Rogers \cite{erdHos1964amount} essentially proved that when $n$ and $N$ are large enough, the best configuration of $\lceil\alpha N\rceil$ (for $\alpha \in (0.99,1.01)$) caps captures roughly $90\%$ of the measure of the sphere. Thus, in high dimensions the overlap is significant.
	\item  The minimal number of caps (with volume $N^{-1}$) required to cover the entire sphere is at most $Cn\ln n \cdot N$ (\cite{rogers1957note,boroczky2003covering}) and at least $cn\cdot N$ (\cite{rogers1963covering}). 
	\item By \cite[Theorem 1.2]{janson1986random}, if we have $C(n)N\ln N$ disjoint balls of measure $N^{-1}$ (where $C(n)$ grows exponentially fast in $n$), then with high probability their union will be the entire sphere $\mathbb{S}^{n-1}$. 
\end{enumerate}

In some sense, Items 1 and 2 are remarkable. Item 1 states that, amazingly, the expected measure of a random partial covering is invariant to the geometry of the sphere. Item 2 states that in high dimensions, the geometry of the sphere causes an ``intrinsic'' overlap of the balls that cannot be avoided, and is independent of the randomness of the points (unlike the case of $\mathbb{S}^{1}$).  Intuitively, one should expect that this overlap will also reduce the expectation as well. From this discussion, the following questions can be raised:

\begin{enumerate}
	\item What is the variance (or a sharp concentration inequality) for the measure of the union of the geodesic balls?
	\item Does the measure of the overlap increase as the dimension grows? Formally, for any fixed $\alpha \in (0,\infty)$ define the sequence   
	\[
C_{n,\alpha}:=\lim_{N\to\infty}\max_{x_1,\ldots,x_{\lceil \alpha N \rceil}}\sigma(\cup_{i=1}^{\lceil \alpha N \rceil}B_G(x_i,N^{-1})), \quad n\in\Bbb N,
	\]
	where $B_G(x_i,N^{-1})$ is the geodesic ball with center $x_i$ and normalized surface measure $N^{-1}$. Is this sequence decreasing in $n$ for every fixed $\alpha$?
\item Is the expectation of the random partial covering optimal as the dimension tends to infinity? More specifically, does it hold that\footnote{From personal communication, Prof. Alexey Glazyrin has also observed this phenomenon.}
\[
\lim_{n \to \infty}C_{n,\alpha} = 1-e^{-\alpha}
\]	
for every fixed $\alpha \in (0,\infty)$?
	
\end{enumerate} 
For the first question, we  give a sharp estimate for the variance, which implies that the deviations decrease as the dimension grows. Also, an immediate use of McDiarmid's inequality yields a concentration inequality. We summarize these results in the next theorem.
\begin{theorem}\label{partialthm}
	Choose $X_1,\ldots,X_{\lceil \alpha N \rceil} \stackrel{\text{i.i.d.}}{\sim} \sigma(\Sn)$,  let $\mathbf{X}_{\alpha N}:=\{X_1,\ldots,X_{\lceil \alpha N \rceil}\}$, and define 
\[
\mathbb{V}(\mathbf{X}_{\alpha N}) :=  \sigma(\cup_{i=1}^{\lceil \alpha N \rceil}B_G(X_i,N^{-1})).
\]
	Assume that $ N \geq C^{n}$. Then for all $\alpha \in (0,\infty)$,
	 \[
	 	e^{-C_1\alpha} c_1^{n-1}N^{-1} \leq \var(\mathbb{V}(\mathbf{X}_{\alpha N})) \leq e^{-\alpha} c_2^{n-1}N^{-1},
	 \]
	  where $c_1,c_2 \in (0,1)$ and $C_1 \in (1,2)$. Moreover, for any $ N \in \mathbb{N}$ and $\epsilon>0$, it holds that 
	 \[
	 	\Pr(|\mathbb{V}(\mathbf{X}_{\alpha N}) - (1-e^{-\alpha})| \geq \epsilon) \leq 2\exp\left(-2\lfloor\alpha^{-1}\rfloor N\epsilon^2\right).
	 \]
\end{theorem}
\noindent{Unfortunately, our concentration inequality is ``dimension-free'', and we don't know how to utilize the dimension in order to achieve the optimal dimensional constant in the exponent. However, our variance estimate utilizes the dimension, and shows that the variance decreases as the dimension grows.}

We don't know how to answer questions 2 and 3 above, which, in our opinion are fundamental. We believe that the answers to these questions will help provide an understanding of the behavior of volumes in high dimensions.  We conclude this section with a conjecture and a question.
\begin{conj}\label{mainconj1}
	For any $\alpha \in (0,\infty)$, the sequence $\{C_{n,\alpha}\}_{n \in \mathbb{N}}$ is decreasing in $n$. Moreover, for every $ \alpha \in (0,\infty)$,
	\[	
		\lim_{n \to \infty}C_{n,\alpha} = \lim_{n \to \infty}\lim_{N \to \infty}\mathbb{V}(\mathbf{X}_{\alpha N}) = 1-e^{-\alpha}.
	\]
\end{conj}
Note that if Conjecture \ref{mainconj1}  holds (say, for $\alpha=1$), then it gives a remarkable geometric definition of  Euler's number $e$, namely,
\begin{equation}\label{eulerconj}
1-e^{-1} = \lim_{n\to\infty}\lim_{N\to\infty}\max_{x_1,\ldots,x_N}\sigma\left(\cup_{i=1}^N B_G(x_i,N^{-1})\right).
\end{equation}
Finally, we mention a perhaps simpler question, which is to utilize the dimension to find a one-sided concentration inequality that improves as the dimension increases, i.e., as $n\to\infty$ the probability of capturing surface area decreases. 
\begin{ques}
		For any $\alpha \in (0,\infty)$ and any $\epsilon > C(n)/\sqrt{N}$, does it hold that
		\[
			\Pr( \mathbb{V}(\mathbf{X}_{\alpha N}) - (1-e^{-\alpha}) \geq \epsilon) \leq 2\exp\left(-C_1(\alpha)C(n)N\epsilon^2\right)?
		\]
\end{ques}

\subsection*{Acknowledgements}
		The authors would like to thank the Mathematical Sciences Research Institute (MSRI) for its hospitality. It was during the authors' visit there when significant progress was achieved. The second-named author wants to thank Prof. Emanuel Milman for his useful suggestions, and also his great advisor Prof. Boaz Nadler and Prof. Yuval Filmus. The authors would also like to thank Prof. Monika Ludwig, Prof. Carsten Sch\"utt and Prof. Elisabeth Werner for the enlightening discussions and insightful suggestions.
		 Finally, we would like to thank the anonymous referees for  thoroughly reading this article, and for their helpful comments and corrections.
		

\subsection*{Background and notation}

Here we provide some background information on convex sets and sphere coverings, and we fix the notation that will be used throughout the paper. 

\vspace{2mm}

The $n$-dimensional volume of a compact set $K\subset \R^n$ is denoted $|K|$. The boundary of $K$ is denoted $\partial K$, and the surface area of $K$ is $|\partial K|$. The Gaussian curvature of $K$ at $x\in\partial K$ is denoted $\kappa(x)$. 
The polar of $K$ is the set $K^\circ:=\{x\in\R^n: \langle x,y\rangle \leq 1, \,\forall y\in K\}$. 

The Euclidean unit ball in $\mathbb R^n$ centered at the origin is denoted by $B_n=\{x\in\R^n: \|x\|\leq 1\}$, where $\|x\|= (\sum_{i=1}^n x_i^2)^{1/2}$ is the Euclidean norm of $x=(x_1,\ldots,x_n)\in\R^n$. The boundary of $B_n$ is the unit sphere $\Sp=\{x\in\R^n: \|x\|=1\}$, and $\sigma$ denotes the uniform probability measure on $\Sp $, i.e., $\sigma(A) = \frac{|A|}{|\partial B_n|}$ for any Borel set $A\subset\Sp$. The volume of the unit ball is $|B_n| = \frac{\pi^{n/2}}{\Gamma(\frac{n}{2}+1)}$, where $\Gamma(x)=\int_0^\infty t^{x-1}e^{-t}\,dt$ denotes the gamma function. In particular, the volume and surface area of the unit ball satisfy the cone-volume formula $|\partial B_n|=n|B_n|$. For more information on convex sets, see, e.g., the monograph of Schneider \cite{SchneiderBook}.

\vspace{2mm}

Let $ K$ be a $ C^2 $ convex body with  positive generalized Gaussian curvature. Consider the metric space $(\partial K, d_G)$, where $d_G$ denotes the geodesic distance on $\partial K$. We shall use the notation $B_G(x,r):=\{y\in\partial K: d_G(x,y)\leq r\}$ to denote the geodesic ball with center $x$ and radius $r$. A finite subset $\mathcal N\subset \partial K$ is called a {\it $\delta$-net} of $\partial K$ if for every $x\in \partial K$ there exists $y\in \mathcal N$ such that $d_G(x,y) \leq \delta$. A $\delta$-net $\mathcal N$ of $\partial K$ is also called a {\it covering} of $\partial K$  since $\cup_{x\in\mathcal N}B_G(x,\delta)\supset \partial K$. We use the term ``covering" in quotes to mean a partial covering, i.e.,  the union of the corresponding geodesic balls is a proper subset of $\partial K$. When $ K = B_n $, a $\delta$-net of $\Sp$ is also called a {\it sphere covering}. 
For more information on sphere coverings, see, e.g., the monograph of B\"or\"oczky \cite{boroczky2004finite}. 

\vspace{2mm}

Throughout the paper, $c,c_1,C,c_0,c_1,c_2,\ldots$ will denote positive absolute constants that may change from line to line. The dependence of a positive constant on  the dimension $n$ or a convex body $K$ or a probability measure $\mu$ will always be stated explicitly as  $ c(n),C(n),c_0(n),c_1(n),\ldots$ or $ c(K), C(K), c_1(K)\ldots $ or $c(\mu), c(\mu,n),\ldots$, respectively; moreover, these constants  may also change from line to line. Finally, please note that all instances of the asymptotic notations $O(\cdot),o(\cdot)$ are meant with respect to the dimension $n$ only.

\section{Auxiliary lemmas}

The following is the classical McDiarmid's inequality \cite{McDiarmid1989}. It is used in the proofs of Theorem \ref{geneThm}, Corollary \ref{mainThm} and Theorem  \ref{partialthm}.
\begin{lemma}[McDiarmid's inequality]
	Let $ X_1,\ldots,X_m $ be independent random variables all taking values in the set $ \mathcal{X}.$ Further, let $ f:\mathcal{X}^m\to \R $ be a
	function of $ X_1,\ldots,X_m $ such that for $1\leq i\leq m$, 
	\begin{align}\label{boundeddiff}
	\sup_{x_1,\ldots,x_m,x_i'\in\mathcal{X}}|f(x_1,\ldots,x_i,\ldots,x_m) - f(x_1,\ldots,x'_i,\ldots,x_m)| \leq c_i .
	\end{align}
	Then
	\[
	\Pr(|f(X_1,\ldots,X_m)-\E[f]| \geq \epsilon) \leq 2\exp\left(-\frac{2\epsilon^2}{\sum_{i=1}^{m}c_i^2}\right).
	\]
\end{lemma}
\vspace{2mm}

We will also need Stirling's inequality (see, e.g., \cite{artin}).

\begin{lemma}[Stirling's inequality]\label{stirlingineq}
	For $x>0$,
	\[x^x e^{-x}\sqrt{2\pi x} < \Gamma(x+1) < x^x e^{-x} e^{\frac{1}{12x}}\sqrt{2\pi x} .\]
\end{lemma}

An immediate consequence of Stirling's inequality is the following well-known lemma.

\begin{lemma}\label{stirling}
	\[
	\frac{\sqrt{2\pi}}{\sqrt{n+2}} \leq \frac{|B_n|}{|B_{n-1}|} \leq  \frac{\sqrt{2\pi}}{\sqrt{n}}
	\]
\end{lemma}



\vskip .1in

For simplicity, we first prove Corollary \ref{mainThm}. This proof captures the main ideas of the proof of Theorem \ref{geneThm}. Later, in Section \ref{pfthm1}, we show how to formally extend this proof from the ball to all smooth convex bodies with positive curvature in order to obtain Theorem \ref{geneThm}.
 
\section{Proof of Corollary \ref{mainThm}}
We first prove the corollary for the uniform measure on the sphere, i.e., $\mu := \sigma$. After this, we extend the proof to an arbitrary density at the end of this section. 

\vskip .1in

By the definitions of $t_{n,N}$ and $P_{n,N}$, we have $t_{n,N}B_n\subset P_{n,N}$. Using this and Fubini's theorem, we express the volume of the symmetric difference $P_{n,N} \triangle B_n$ as
\begin{equation}\label{Eq:refine}
\begin{aligned}|P_{n,N}\triangle B_n| & =\int_{\R^{n}\setminus B_n}\mathbbm{1}_{\{\max\langle X_{i},x\rangle\leq t_{n,N}\}}(x)\,dx+\int_{B_n\setminus t_{n,N}B_n}(1-\mathbbm{1}_{\{\max\langle X_{i},x\rangle\leq t_{n,N}\}}(x))\,dx\\
& =|\partial B_n|\bigg(\int_{\mathbb{S}^{n-1}}\int_{1}^{\infty}r^{n-1}\mathbbm{1}_{\{\max\langle X_{i},x\rangle\leq r^{-1}t_{n,N}\}}(x)\,dr\,d\sigma(x)\\
& +\int_{\mathbb{S}^{n-1}}\int_{t_{n,N}}^{1}r^{n-1}(1-\mathbbm{1}_{\{\max\langle X_{i},x\rangle\leq r^{-1}t_{n,N}\}}(x))\,dr\,d\sigma(x)\bigg)\\
& =|\partial B_n|\bigg(\int_{t_{n,N}}^{1}r^{n-1}\int_{\Sp}(1-\mathbbm{1}_{\{\max\langle X_{i},x\rangle\leq r^{-1}t_{n,N}\}}(x))\,d\sigma(x)\,dr\\
& +\int_{1}^{1+c_{n,N}}r^{n-1}\int_{\mathbb{S}^{n-1}}\mathbbm{1}_{\{\max\langle X_{i},x\rangle\leq r^{-1}t_{n,N}\}}(x)\,d\sigma(x)\,dr\\
& +\int_{1+c_{n,N}}^{1+\frac{2}{n}}r^{n-1}\int_{\mathbb{S}^{n-1}}\mathbbm{1}_{\{\max\langle X_{i},x\rangle\leq r^{-1}t_{n,N}\}}(x)\,d\sigma(x)\,dr\\
& +\int_{1+\frac{2}{n}}^{\infty}r^{n-1}\int_{\mathbb{S}^{n-1}}\mathbbm{1}_{\{\max\langle X_{i},x\rangle\leq r^{-1}t_{n,N}\}}(x)\,d\sigma(x)\,dr\bigg)\\
& =Y+Z+W.
\end{aligned}
\end{equation}
Here $ c_{n,N}:=c(n)\NN(\ln N)^{\frac{2}{n-1}} $ and $c(n)$ will be defined at the end of Subsection \ref{Znegligiblesection}, and  
\begin{equation}\label{Eq:refine2}
\begin{aligned}Y & :=|\partial B_n|\bigg(\int_{1}^{1+c_{n,N}}r^{n-1}\int_{\mathbb{S}^{n-1}}\mathbbm{1}_{\{\max\langle X_{i},x\rangle\leq r^{-1}t_{n,N}\}}(x)\,d\sigma(x)\,dr\\
& +\int_{t_{n,N}}^{1}r^{n-1}\int_{\mathbb{S}^{n-1}}(1-\mathbbm{1}_{\{\max\langle X_{i},x\rangle\leq r^{-1}t_{n,N}\}}(x))\,d\sigma(x)\,dr\bigg) \\
Z & :=|\partial B_n|\int_{1+c_{n,N}}^{1+\frac{2}{n}}r^{n-1}\int_{\mathbb{S}^{n-1}}\mathbbm{1}_{\{\max\langle X_{i},x\rangle\leq r^{-1}t_{n,N}\}}(x)\,d\sigma(x)\,dr\\
W & :=|\partial B_n|\int_{1+\frac{2}{n}}^{\infty}r^{n-1}\int_{\mathbb{S}^{n-1}}\mathbbm{1}_{\{\max\langle X_{i},x\rangle\leq r^{-1}t_{n,N}\}}(x)\,d\sigma(x)\,dr
=|P_{n,N}\cap((1+\tfrac{2}{n})B_n)^{c}|.
\end{aligned}
\end{equation}

We split the proof of Corollary \ref{mainThm} into three lemmas, one for each random variable $Y$, $Z$ and $W$. First, we apply McDiarmid's inequality 
to  derive a concentration inequality for $Y$. For $Z$, we apply 
a ``pigeonhole principle" to certain  coverings of the sphere to show that $Z$ is negligible with high probability. Finally, we use a standard sphere covering argument to show that $W=0$ with extremely high probability. In the final analysis, we condition on the event that $Z+W$ is extremely small to derive the desired concentration inequality for $| P_{n,N} \triangle B_n|$. 


\subsection{Concentration for the random variable $Y$}

\begin{lemma}\label{Yprop}
	Let $\epsilon>0$. There is a constant $c_1(n)>0$ such that for all sufficiently large $N$,
	\[
	\Pr(|Y-\E[Y]| \geq \epsilon) \leq
	\exp\left(-c_1(n)N^{1+\frac{4}{n-1}}f(N)\epsilon^2\right)
	\]
	where $f(N)=(\ln N)^{-(2+\frac{4}{n-1})}$.
\end{lemma}

\begin{proof}
     We aim to apply McDiarmid's inequality to the  random variables
	\begin{align*}
	Y_1(X_1,\dots,X_N) &:=|P_{n,N}^c\cap(1+c_{n,N})B_n\cap B_n^c| \\
	Y_2(X_1,\dots,X_N) &:=|B_n \setminus P_{n,N} |,
	\end{align*}
	where $Y = \left(|(1+c_{n,N})B_n\cap B_n^c| - Y_1\right) +Y_2$. For all $1\leq i\leq N$, we have 
%
	\begin{align*}
	c_i:&=\sup_{x_1,\ldots,x_N,x'_i}|Y_1(x_1,\ldots,x_i,\ldots ,x_N) - Y_1(x_1,\ldots,x'_i,\ldots x_N)| \\
	&\leq \sup_{x_1,\ldots,x_N,x'_i} |(x_i^\perp)^{-}\cap (1+c_{n,N})B_n|,
	\end{align*}
	where $x_i^\perp$ is the hyperplane orthogonal to $x_i$ that contains $x_i$, $(x_i^\perp)^{-}$ is the halfspace of $x_i^\perp$ that does not contain the origin, and $(x_i^\perp)^{-}\cap (1+c_{n,N})B_n$ is the cap of $(1+c_{n,N})B_n$ with height $1+c_{n,N}-t_{n,N}$ and base $x_i^\perp\cap(1+c_{n,N})B_n$.
	We can bound the volume of each cap by 
	\begin{align*}|(x_{i}^{\perp})^{-}\cap(1+c_{n,N})B_n| & =|B_{n-1}|\int_{t_{n,N}}^{1+c_{n,N}}\left((1+c_{n,N})^{2}-x^{2}\right)^{\frac{n-1}{2}}\,dx\\
	& \leq2^{\frac{n-1}{2}}(1+c_{n,N})^{\frac{n-1}{2}}|B_{n-1}|\int_{t_{n,N}}^{1+c_{n,N}}(1+c_{n,N}-x)^{\frac{n-1}{2}}\,dx\\
	&=c(n)(1+c_{n,N}-t_{n,N})^{\frac{n+1}{2}}\\
	&\leq c_{1}(n)N^{-(1+\frac{2}{n-1})}(\ln N)^{1+\frac{2}{n-1}}.
	\end{align*}
	The last inequality follows since
	\begin{align*}
	1+c_{n,N}-t_{n,N} &\leq 1+c(n)\NN(\ln N)^{\frac{2}{n-1}}-\left[1-\frac{1}{2}\left(1+\frac{c\ln n}{n}\right)\NN\right]\\
	&\leq 2c(n)\NN(\ln N)^{\frac{2}{n-1}}.
	\end{align*}
	Thus, from McDiarmid's inequality it follows that for any $\epsilon > 0$,
	\begin{align}\label{Yconcentration}
	\Pr(|Y_1-\E[Y_1]| \geq \epsilon|B_n|) &\leq 2\exp\left(-\frac{2\epsilon^2}{4 N c(n)^2 N^{-(2+\frac{4}{n-1})}(\ln N)^{2+\frac{4}{n-1}}}\right) \nonumber\\
	&= 2\exp\left(-c_1(n)N^{1+\frac{4}{n-1}}f(N)\epsilon^2\right)
	\end{align}where $f(N)=(\ln N)^{-(2+\frac{4}{n-1})}$.
	The second variable $Y_2$   
	is handled in a similar way, and it can be derived that 
	\begin{equation}\label{concY2}
	\Pr(|Y_2-\E[Y_2]| \geq \epsilon |B_n|) \leq 2\exp\left(-c_2N^{1+\frac{4}{n-1}}\epsilon^2\right).
	\end{equation}
	Finally, we apply a union bound and use Eqs. (\ref{Yconcentration}) and (\ref{concY2}) to derive the desired concentration inequality for $Y$.
\end{proof}


\subsection{The random variable $Z$ is negligible}\label{Znegligiblesection}

First, we mention that McDiarmid's inequality does not provide the desired concentration inequality for $ Z $ because the volume of each geodesic ball is too big. Furthermore, despite the fact that the expectation of $Z$ is extremely small, other standard concentration inequalities  do not yield the desired inequality for $Z$. We  instead use a direct geometric and combinatorial argument to show that $ Z $ is negligible with high probability. 

\begin{lemma}\label{Zprop}
	For all sufficiently large $N$,
	\[\Pr\left(Z \geq N^{-(0.5+\frac{2}{n-1})}\right) \leq \exp\left(-c_2(n)N^{0.5-\frac{2}{n-1}}\ln N\right).\]
\end{lemma}
\begin{proof}
	Recall that $t_{n,N}:=\sqrt{1-\left(\frac{(\ln 2)|\partial B_n|}{N|B_{n-1}|}\right)^{\frac{2}{n-1}}}$. Since $t_{n,N}<1$,  for all $x\in\Sp$ and all $r> 0$ it holds that 
	\[
	\mathbbm{1}_{\max\langle X_i,x\rangle \leq r^{-1}t_{n,N}} \leq \mathbbm{1}_{\max\langle X_i,x\rangle \leq r^{-1}}.
	\]
	Therefore, 
	\begin{align*} Z &\leq |\partial B_n|\int_{1+c_{n,N}}^{1+\frac{2}{n}}r^{n-1}\int_{\mathbb{S}^{n-1}}\mathbbm{1}_{\{\max\langle X_{i},x\rangle\leq r^{-1}\}}(x)\,d\sigma(x)\,dr\\
	& \leq|\partial B_n|\int_{1+c_{n,N}}^{1+\frac{2}{n}}r^{n-1}\int_{\mathbb{S}^{n-1}}\mathbbm{1}_{\{\max\langle X_{i},x\rangle\leq(1+c_{n,N})^{-1}\}}(x)\,d\sigma(x)\,dr\\
	& =|\partial B_n|\left[\frac{r^{n}}{n}\right]_{1+c_{n,N}}^{1+\frac{2}{n}}\times\int_{\Sp}\mathbbm{1}_{\{\max\langle X_{i},x\rangle\leq(1+c_{n,N})^{-1}\}}(x)\,d\sigma(x)\\
	& \leq3|B_n|\int_{\Sp}\mathbbm{1}_{\{\max\langle X_{i},x\rangle\leq(1+c_{n,N})^{-1}\}}(x)\,d\sigma(x)=\frac{3}{n}\tilde{Z},
	\end{align*}
	where $\tilde Z:=|\partial B_n|\int_{\mathbb{S}^{n-1}}\mathbbm{1}_{\{\max\langle X_{i},x\rangle\leq(1+c_{n,N})^{-1}\}}(x)\,d\sigma(x)$. 
	The random variable $ \tilde Z $ has a geometric meaning: it measures the missing surface area of a random sphere ``covering" with $ N $ random geodesic balls of measure $ c(n)\frac{\ln N}{N}$. Observe that the base of each spherical  cap of the ``covering" 
	\[
	\textstyle \Sp \cap \left(\bigcap_{i=1}^N \{x\in\R^n: \langle X_i,x\rangle \leq (1+c_{n,N})^{-1}\}\right)^c
	\]
	has radius $\sqrt{1-(1+c_{n,N})^{-2}}>  \sqrt{c_{n,N}}$, where we used the inequality $(1+x)^{-2} < 1-x$ for $x\in(0, (\sqrt{5}-1)/2)$. 
	
	To get some intuition, we first show that the expectation of $ \tilde Z $ is extremely small. By  independence and Lemma \ref{stirling},
	\begin{align*}
	\E[\tilde Z] &=\int_{\Sp} \E[\mathbbm{1}_{\max\langle X_i,x\rangle \leq (1+c_{n,N})^{-1}}(x)]\,d\sigma(x)\\
	&=\int_{\Sp} \prod_{i=1}^N\Pr(\langle X_i,x\rangle \leq (1+c_{n,N})^{-1})\,d\sigma(x)\\
	&<\left(1-\frac{c_{n,N}^{\frac{n-1}{2}}|B_{n-1}|}{|\partial B_n|}\right)^N
	\leq \left(1-\frac{c(n)^{\frac{n-1}{2}}\ln N}{\sqrt{2\pi n}N}\right)^N\\
	&\leq \exp\left(-\frac{c(n)^{\frac{n-1}{2}}}{\sqrt{2\pi n}}\cdot\ln N\right)=N^{-c(n)^{\frac{n-1}{2}}/\sqrt{2\pi n}},
	\end{align*}where  we used the elementary inequality $(1-x)^N \leq e^{-Nx}$ for $x\in(0,1)$. 
	At the end of the proof of the lemma, we will choose $ c(n) $ to be large enough so that this expectation is negligible.
	
	Recall that our goal is to show that with high probability, $ \tilde{Z} $ is negligible. We can reduce the problem to the following  random sphere ``covering" of $ c(n)N\ln N$ random geodesic balls with volume $ N^{-1}.$	Let $ \mathbf{X} := \{X_1,\ldots, X_{c(n)N\ln N}\} $ and define 
	\[
	A_{\mathbf X} :=\bigcup_{i=1}^{c(n)N\ln N} B_{G}(X_i,N^{-1})
	\]
	to be the random ``covering" generated by $\mathbf{X}$. We want to estimate the probability of the  event 
	\[
	B :=\{ |A^c_{\mathbf X}| = \tilde Z > |\partial B_n|N^{-(0.5+\frac{2}{n-1})}\}.
	\]
	To do so, we will need a sphere covering such that each point of the sphere is not counted ``too many" times, i.e., each point must not belong to ``too many" balls in the covering. A remarkable result of B\"or\"oczky and Wintsche \cite[Theorem 1.1]{boroczky2003covering} provides such a covering,  showing that there exists a $ \frac{1}{2}N^{-\frac{1}{n-1}}$-net  
	$\mathcal{N}$ of size $400 n\ln n\cdot 2^n N$ 
	such that each point $x\in\Sn$ lies inside of at most $ 400 n\ln n$ caps. 
	
	Next, we use the covering $\mathcal N$ to define the random  set 
	\[S_{\max} := \left\{y\in\mathcal N: \, d_G(X_i,y) \geq \tfrac{1}{2}N^{-\frac{1}{n-1}} , \, \forall 1 \leq i\leq c(n)N\ln N \right\}.\] 
	Observe that if $ y \notin S_{\max} $, then by the triangle inequality $ B_G(y, 2^{-(n-1)}N^{-1}) \subset A_{\mathbf X}$, and hence
	\[
	A_{\mathcal N\setminus S_{\max}} := \bigcup_{y\in\mathcal N\setminus S_{\max}} B_G(y, 2^{-(n-1)}N^{-1}) 
	\subset A_{\mathbf X}.
	\] 
	We claim that this inclusion implies  $ |S_{\max}| \geq 2^{n-1}N^{0.5-\frac{2}{n-1}}$ holds under the event $ B.$ To see this, we prove the contrapositive statement and assume that $|S_{\max}| < 2^{n-1}N^{0.5-\frac{2}{n-1}}$. Then 
	\begin{align*} 
	|A^c_{\mathbf X}| \leq |A^c_{\mathcal N\setminus S_{\max}}|
	\leq |A_{S_{\max}}| 
	&= \big|\bigcup_{ y \in S_{\max}}B_G(y,2^{-(n-1)}N^{-1})\big| \\
	&= |S_{\max}| \cdot 2^{-(n-1)}N^{-1}|\partial B_n| \\
	&< N^{-(0.5+\frac{2}{n-1})}|\partial B_n|,
	\end{align*} 
	so the complementary event $B^c$ holds. 
	Therefore,
	\begin{align}\label{starineq}
	\Pr_{\mathbf X}(B) &\leq \Pr_{\mathbf X}(|S_{\max}| \geq   2^{n-1}N^{0.5-\frac{2}{n-1}}) \nonumber
	\\&= \Pr_{\mathbf X}(\exists S \subset S_{\max}: \, |S| = 2^{n-1}N^{0.5-\frac{2}{n-1}}).
	\end{align}
Since each point of the sphere lies inside of at most $400 n\ln n$ balls of the covering $\mathcal{N}$,
	\begin{align}\label{Asmaxlower}
	|A_{S_{\max}}| = \int_{\Sn}\mathbbm{1}_{A_{S_{\max}}}(x)\,dS(x) 
	&\geq \frac{1}{400 n\ln n}\int_{\Sn}\sum_{y\in S_{\max}}\mathbbm{1}_{B_G(y,2^{-(n-1)}N^{-1})}(x)\,dS(x) \nonumber \\
	&=\frac{1}{400 n\ln n}|S_{\max}|\cdot|B_G(y,2^{-(n-1)}N^{-1})| \nonumber\\
	&\geq\frac{2^nN^{-(0.5+\frac{2}{n-1})}}{400 n\ln n}|\partial B_n|.
	\end{align} 
	Thus, by (\ref{starineq}), (\ref{Asmaxlower}), a union bound and independence, we conclude that
	\begin{equation}\label{Zfinal}
	\begin{aligned}
	\Pr_{\mathbf X}(B) &\leq \Pr_{\mathbf X}(\exists S\subset S_{\max}: |S|=2^{n-1}N^{0.5-\frac{2}{n-1}}) \\
	&\leq \binom{400 n\ln n\cdot2^nN}{ 2^nN^{0.5-\frac{2}{n-1}}}\cdot\Pr_{\mathbf X}(X_{1},\ldots, X_{c(n)N\ln N}\notin A_{S_{\max}})
	\\& \leq \exp\left(c_2(n)(\ln N)N^{0.5-\frac{2}{n-1}}\right)\cdot\left(1-\frac{2^nN^{-(0.5+\frac{2}{n-1})}}{400 n\ln n}\right)^{c(n)N\ln N} 
	\\&\leq \exp\left(c_2(n)(\ln N)N^{0.5-\frac{2}{n-1}}-c(n)c_1(n)(\ln N)N^{0.5-\frac{2}{n-1}}\right).
	\end{aligned}
	\end{equation}
	Choosing $ c(n) $  large enough yields the lemma.
\end{proof}

\subsection{The random variable $W$ equals zero with high probability}

Finally, we turn our attention to the random variable $W=|P_{n,N}\cap ((1+\frac{2}{n})B_n)^c|$. The next lemma implies that $W=0$ with high probability. 

\begin{lemma}{\label{Wconcentration}}
	When $N$ is large enough, the polytope $ P_{n,N} $ lies inside the ball $ (1+\frac{2}{n})B_n $ with probability at least $1-e^{-c_1(n)N}.$
\end{lemma}



We break the proof of Lemma \ref{Wconcentration} into two steps. First, we show that if $\mathcal N$ is a $\frac{1}{\sqrt{n}}$-net of the sphere, then the inclusion $P_{n,N} \subset (1+\frac{2}{n})B_n$ holds. In the second step, we show that if the points $X_1,\ldots, X_N$ are chosen uniformly and independently from $\Sp$, then the random set $\{X_1,\ldots,X_N\}$ is a $\frac{1}{\sqrt{n}}$-net of $\Sp$ with probability at least $1-e^{-c_1(n)N}$. 



\begin{lemma}\label{lemma1W}
	Suppose that $\mathcal N=\{X_1,\ldots,X_N\}$ is a $\frac{1}{\sqrt{n}}$-net of $\Sp$. Then $P_{n,N} \subset (1+\frac{2}{n})B_n$.
\end{lemma}

\begin{proof}
	Suppose by way of contradiction that there exists a point $v\in P_{n,N}$ such that $\|v\| \geq 1+\frac{2}{n}$. Without loss of generality, we may assume that  $v = (1+\frac{2}{n})e_1$. We claim that if $v\in P_{n,N}$, then $P_{n,N}$ has no facet $\{\langle X_i,x\rangle = t_{n,N}\}$ with $\langle X_i, e_1\rangle > 1-\frac{1}{n}$. Otherwise, if such a facet exists then there is an index $j$ such that
	\[
	\langle X_j, v\rangle = \langle X_j, e_1\rangle v_1 
	\geq \left(1-\frac{1}{n}\right)\left(1+\frac{2}{n}\right) = 1+\frac{1}{n}-\frac{2}{n^2} > t_{n,N},
	\]which is a contradiction. Thus, none of the outer normals $X_1,\ldots,X_N$ of the facets of $P_{n,N}$ lie in the cap $ \{x\in\Sp: \langle x, e_1\rangle > 1-\frac{1}{n}\}$. Let $r$ denote the radius of this cap. Then for all $n\geq 2$,
	\[
	r= \sqrt{1-\left(1-\frac{1}{n}\right)^2} = \sqrt{\frac{2}{n}-\frac{1}{n^2}} = \frac{\sqrt{2}}{\sqrt{n}}\cdot\sqrt{1-\frac{1}{2n}} \geq \frac{\sqrt{3}}{\sqrt{2n}} > \frac{1}{\sqrt{n}}.
	\]
	This implies that $\mathcal N$ is not a $\frac{1}{\sqrt{n}}$-net, a contradiction.
\end{proof}

\begin{lemma}\label{deltanet}
	Let $X_1,\ldots,X_N \stackrel{\text{i.i.d.}}{\sim} \sigma$. For all sufficiently large $N$, the set $\mathcal N:=\{X_1,\ldots,X_N\}$ is a $ \frac{1}{\sqrt{n}}$-net of the unit sphere with probability at least $ 1-e^{-c_1(n)N}.$ 
\end{lemma}
\begin{proof}
	By definition, $\mathcal N$ is a $\frac{1}{\sqrt{n}}$-net of $\Sp$ if and only if for any $x\in\Sp$ there exists $X_i\in\mathcal N$ such that $d_G(x,X_i)\leq \frac{1}{\sqrt{n}}$. We estimate the probability that $\mathcal N$ is not a $\frac{1}{\sqrt{n}}$-net, which holds if and only if  there exists $z\in\Sp$ such that for all $1\leq i\leq N$ we have $d_G(z,X_i)>\frac{1}{\sqrt{n}}$.  By independence and Lemma \ref{stirling}, for any fixed  $ z\in\Sn $ and all sufficiently large $N$ we obtain
	\begin{align}\label{deltanet1}
	\Pr\left(\forall X_i\in\mathcal{N} :\ d_{G}(X_{i},z)>\frac{1}{2\sqrt{n}}\right) & \leq\left(1-\frac{(\frac{1}{2\sqrt{n}})^{n-1}|B_{n-1}|}{|\partial B_n|}\right)^{N} \nonumber\\
	&\leq \left(1-\frac{1}{(\sqrt{\pi/2})^{n-1}n^{n/2}}\right)^N \nonumber\\
	& \leq e^{-c(n)N}.
	\end{align}
	Now using \cite[Cor. 5.5]{aubrun2017alice}, we can find a (deterministic) $ \frac{1}{2\sqrt{n}}$-net of $ \Sn $ of size \\
	$ c_1 n\ln n\cdot\frac{|\partial B_n|}{|C(x,\frac{1}{2\sqrt{n}})|}$; denote it by $ \mathcal{N}_0.$ Therefore, when $N$ is large enough, we can bound the probability that $\mathcal N$ is not a $\frac{1}{\sqrt{n}}$-net by
	\begin{align*}
	\Pr\bigg(\exists x\in\mathbb{S}^{n-1}\,\text{s.t.}\,\forall X_i\in\mathcal N:\ d_{G}(X_{i},x)&>\frac{1}{\sqrt{n}}\bigg) \\
	& \leq\Pr\left(\exists z\in\mathcal{N}_0\,\,\text{s.t.}\,\forall X_i\in\mathcal N:\ d_{G}(z,X_i)>\frac{1}{2\sqrt{n}}\right)\\
	& \leq|\mathcal{N}_0|\Pr\left(\forall X_i\in\mathcal N:\,d_{G}(z,X_{i})>\frac{1}{2\sqrt{n}}\right)\\
	& \leq c_1(n)e^{-N}\\
	& \leq e^{-c_2(n)N}.
	\end{align*}The first inequality follows from the triangle inequality, the second from a union bound and the third from (\ref{deltanet1}). The claim follows.
\end{proof}


\subsection{Concluding Corollary \ref{mainThm}}\label{proofsectionmainthm}

Finally, we put everything together to prove Corollary \ref{mainThm}. Recall that $|P_{n,N} \triangle B_n| = Y+Z+W.$  We  use  the following ingredients:
\begin{enumerate}
	\item[1.] $ \E[|P_{n,N} \triangle B_n|] = c(n)\NN|B_n|$ (see \cite[Theorem 2.1]{kur2017approximation})
	\item[2.] $ \E[Z] \leq \frac{3}{n}\E[\tilde{Z}] = N^{-c(n)^{\frac{n-1}{2}}/\sqrt{2\pi n}} $
	\item[3.] $ \E[W] = 0$ with probability $ 1-e^{c(n)N} $. 
\end{enumerate}	
By  Lemmas  \ref{Zprop} and \ref{Wconcentration} as well as Items 2 and 3,  
\begin{align*}
\Pr(|Z+W| \geq N^{-(0.5+\frac{2}{n-1})}\pm \E[Z]|W = 0) + \Pr(W \ne 0) &\leq\exp\left(-c_1(n)N^{0.5-\frac{2}{n-1}}\right).
\end{align*}
Conditioning on the event $W=0$ and using Lemma \ref{Yconcentration} again, we derive that
\begin{align*}
&\Pr\left(\big||P_{n,N}\triangle B_n| - \E[|P_{n,N}\triangle B_n|]\big| \geq \epsilon \right) \leq
\\&\Pr\left(\left\{\big||P_{n,N}\triangle B_n| - \E[|P_{n,N}\triangle B_n|]\big| \geq \epsilon\right\} \cup \{W \ne 0\} \right) \leq
\\& \Pr\left(\big|Y - \E[Y]\big| \geq \epsilon| W=0 \right) + \Pr\left(|Z| \geq c_1(n)N^{-(0.5+\frac{2}{n-1})} \pm \E[Z] \big| W=0\right) +\Pr(W \ne 0)
\\& \leq 2\exp\left(-c(n)N^{1+\frac{4}{n-1}}f(N)\epsilon^2\right)+\exp\left(-c_1(n)N^{0.5-\frac{2}{n-1}}\right).
\end{align*}
Observe that for $ \epsilon < N^{-(0.5+\frac{2}{n-1})},$ our concentration inequality doesn't give something meaningful. Thus, we can assume that this inequality holds for all $ \epsilon > 0.$
\qed


\subsection*{Extending to an arbitrary density.}
We now describe how to extend the proof of Corollary \ref{mainThm} from the uniform measure $\sigma$ to any probability measure $\mu$ on $\Sp$ for which there exists a density $g$ such that $d\mu(x) = g(x)\,d\sigma(x)$ and $g(x)>0$ for all $x\in\Sp$. In this case, $|P_{n,N}\triangle B_n| = Y_\mu + Z_\mu + W_\mu$, where 
\begin{align*}Y_{\mu} & :=|\partial B_n|\bigg(\int_{1}^{1+c_{n,\mu,N}}r^{n-1}\int_{\mathbb{S}^{n-1}}\mathbbm{1}_{\{\max\langle X_{i},x\rangle\leq r^{-1}t_{n,N}\}}(x)\,g(x)\,d\sigma(x)\,dr\\
&  +\int_{t_{n,N}}^{1}r^{n-1}\int_{\mathbb{S}^{n-1}}(1-\mathbbm{1}_{\{\max\langle X_{i},x\rangle\leq r^{-1}t_{n,N}\}}(x))\,g(x)\,d\sigma(x)\,dr\bigg)\\ 
Z_{\mu} & :=|\partial B_n|\int_{1+c_{n,\mu,N}}^{1+\frac{2}{n}}r^{n-1}\int_{\mathbb{S}^{n-1}}\mathbbm{1}_{\{\max\langle X_{i},x\rangle\leq r^{-1}t_{n,N}\}}(x)\,g(x)\,d\sigma(x)\,dr\\
W_{\mu} & :=|\partial B_n|\int_{1+\frac{2}{n}}^{\infty}r^{n-1}\int_{\mathbb{S}^{n-1}}\mathbbm{1}_{\{\max\langle X_{i},x\rangle\leq r^{-1}t_{n,N}\}}(x)\,g(x)\,d\sigma(x)\,dr.%
\end{align*}
Here $ c_{n,\mu,N}:=c(n,\mu)\NN(\ln N)^{\frac{2}{n-1}} $ and $ c(n,\mu) $ is a large constant that will be defined later.

Since McDiarmid's inequality holds for any probability measure, it follows that for any $\epsilon>0$,
\[
\Pr(|Y_\mu-\E[Y_\mu]| \geq \epsilon) \leq
\exp\left(-c_1(n)N^{1+\frac{4}{n-1}}f(N)\epsilon^2\right)
\]
where $f(N)=(\ln N)^{-(2+\frac{4}{n-1})}$. Thus, we only need to discuss how to extend the proofs of Lemmas \ref{Zprop} and \ref{Wconcentration} to the random variables $Z_\mu$ and $W_\mu$, respectively.

Following the proof of Lemma \ref{Zprop},  observe that the only time  the probability measure was used was in  (\ref{Zfinal}). Since
\begin{equation}\label{musigma}
\mu(A_{S_{\max}})=\int_{A_{S_{\max}}} g(x)\,d\sigma(x) \geq \sigma(A_{S_{\max}})\min_{x\in\Sp}g(x),
\end{equation}
we can modify \eqref{Zfinal} to get
\begin{align*}\Pr(B)&\leq\binom{400n\ln n\cdot2^{n}N}{2^{n}N^{0.5-\frac{2}{n-1}}}\Pr_{X_{i}\sim\mu}(X_{1},\ldots,X_{c(n,\mu)N\ln N}\not\in A_{S_{\max}})  \\&=\binom{400n\ln n\cdot2^{n}N}{2^{n}N^{0.5-\frac{2}{n-1}}}\left(1-\mu(A_{S_{\max}})\right)^{c(n,\mu)N\ln N}\\
& \leq\binom{400n\ln n\cdot2^{n}N}{2^{n}N^{0.5-\frac{2}{n-1}}}\left(1-\frac{2^{n}N^{-(0.5+\frac{2}{n-1})}\min_{x\in\Sp}g(x)}{Cn\ln n}\right)^{c(n,\mu)N\ln N}\\
& \leq\exp\left(c_2(n)(\ln N)N^{0.5-\frac{2}{n-1}}-c(n,\mu)c(g,\mu)(\ln N)N^{0.5-\frac{2}{n-1}}\right).
\end{align*}
Thus, choosing $c(n,\mu)$ large enough shows that $Z_\mu$ is negligible with high probability.

Similarly, in the proof of Lemma \ref{Wconcentration} we only used the probability measure once, in (\ref{deltanet1}). Thus, by independence and (\ref{musigma}), for any fixed $z\in\Sp$ and all sufficiently large $N$ we derive
\begin{align*}
\Pr_{X_i\sim\mu}\left(\forall X_i\in\mathcal{N} :\ d_{G}(X_{i},z)>\frac{1}{\sqrt{n}}\right) &\leq \left(1-\mu(B_G(X_i,\tfrac{1}{\sqrt{n}}))\right)^{N} \\
&\leq \left(1-\sigma(B_G(X_i,\tfrac{1}{\sqrt{n}}))\cdot\min_{x\in\Sp}g(x)\right)^N\\
&\leq e^{-c(n)c(g)N},
\end{align*}where $c(g)$ is a positive constant that depends only on $g$. The rest of the proof for $W_\mu$ is similar to that of Lemma \ref{Wconcentration}.

\section{Proof of Theorem \ref{deldiv}}
Recall  we aim to show that 
\[
|\divv_{n-1} - (2\pi e)^{-1}(n+\ln n)| =O(1).
\]
First, let us prove that  $\divv_{n-1} \leq (2\pi e)^{-1}(n+\ln n) + c.$  By \cite[Theorem 2.1]{kur2017approximation}, there is a polytope $ P $ with $ N $ facets, all of which have the same height $ t_{n,N} $, such that when $N$ is sufficiently large
\[
|P \triangle B_n| \leq C|B_n|\NN.
\]
Now we inflate the polytope $P$ by a factor of $ t_{n,N}^{-1} $ to get a polytope $\tilde P:=t_{n,N}^{-1}P$ that circumscribes $B_n$ (i.e., $\tilde P\supset B_n$ and each facet of $\tilde P$ touches $\partial B_n$). By the homogeneity of volume,
\begin{align*}
|\tilde{P} \setminus B_n|&=|t_{n,N}^{-1}P| - |B_n|\\
& \leq \left(1+C\NN\right)t_{n,N}^{-n}|B_n|-|B_n|\\
& \leq\frac{1}{2}(n+\ln n+c_{0})|B_n|\NN.
\end{align*}
From this and a result of Gruber \cite[Eq. (4)]{GruberOut}, for all sufficiently large  $N$  we obtain
\begin{align*}
&\phantom{=}\,\,\,\frac{1}{2}\left(1-\frac{1}{n^2}\right)\divv_{n-1}|\partial B_n|^{\frac{n+1}{n-1}}\NN\\
&=\frac{1}{2}\left(n+2\ln n +O\bigg(\frac{(\ln n)^2}{n}\bigg)\right)\divv_{n-1}|B_n|^{\frac{n+1}{n-1}}\NN
\\& \leq |\tilde P\setminus B_n| \leq \frac{1}{2}(n+\ln n+c_0)|B_n|\NN,
\end{align*}which implies that
\begin{align}\label{ballvolume1}
\divv_{n-1} &\leq |B_n|^{-\frac{2}{n-1}}\cdot\frac{n+\ln n+c_0}{n+2\ln n +O\bigg(\frac{(\ln n)^2}{n}\bigg)}\nonumber
\\&=
\frac{n+\ln n+c_0}{n+2\ln n +O\bigg(\frac{(\ln n)^2}{n}\bigg)}\left(\frac{{\pi}^{\frac{n}{2}}}{\Gamma(\frac{n}{2}+1)}\right)^{-\frac{2}{n-1}}.
\end{align}
From Lemma \ref{stirlingineq} (Stirling's inequality) we obtain the estimate
\begin{align}\label{ballvolume2}
\left(\frac{{\pi}^{\frac{n}{2}}}{\Gamma(\frac{n}{2}+1)}\right)^{-\frac{2}{n-1}} &= \pi^{-1}(1+O(\tfrac{1}{n}))\left(\Gamma\left(\frac{n}{2}+1\right)\right)^{\frac{2}{n-1}} \nonumber\\
&=\pi^{-1}(1+O(\tfrac{1}{n}))\left(\sqrt{2\pi\cdot\frac{n}{2}}\left(\frac{n}{2}\right)^{\frac{n}{2}}e^{-\frac{n}{2}}\right)^{\frac{2}{n-1}}\nonumber\\
&=(2\pi e)^{-1}(n+2\ln n+O(1)).
\end{align}
Combining (\ref{ballvolume1}) and (\ref{ballvolume2}) yields
\begin{align}\label{divleq}
\divv_{n-1} &\leq (2\pi e)^{-1}(n+2\ln n+O(1))\cdot\frac{n+\ln n+c_0}{n+2\ln n +O\bigg(\frac{(\ln n)^2}{n}\bigg)} 
\nonumber\\&= (2\pi e)^{-1}(n+2\ln n+O(1))\left(1-\frac{\ln n+O(1)}{n+2\ln n +O\bigg(\frac{(\ln n)^2}{n}\bigg)}\right)
\nonumber\\& = (2\pi e)^{-1}(n+\ln n)+O(1).
\end{align}



\vspace{2mm}

In the other direction, we show that $\divv_{n-1} \geq (2\pi e)^{-1}(n+\ln n) - c_1$, where $ c_1 > 0  $ is an absolute constant that will be defined later. Suppose that there exists a polytope $P_b \supset B_n$ with $N$ facets such that \[ |P_b| \leq \left(1+\frac{1}{2}(n + \ln n- c_1)\NN\right)|B_n|.\]   Without loss of generality, we may assume that all of the facets of $P_b$ touch the unit ball.
Now  shrink $P_b$ so that its volume equals $ |B_n|$, and  denote the resulting shrunken polytope by $ \hat{P_b}.$ Then $\hat{P_b}$ can be represented as 
\[
\hat{P_b} = \bigcap_{i=1}^{N}\{x\in\R^{n}:\ \inneri{y_i}{x} \leq w_{n,N}\},
\] where $y_1,\dots,y_N$ are the normals of the facets of $\hat{P_b}$ and $w_{n,N} \geq t_{n,N} +\frac{c_1}{2n}\NN$. We express the volume of $B_n\setminus \hat{P_b}$ as
\begin{align}\label{eq:a1}
|B_n \setminus \hat{P_b}| &= \left|\cup_{i=1}^{N}\{x\in\R^{n}:\inneri{y_i}{x} > w_{n,N} \} \cap B_n\right|  \nonumber\\
&\leq N|\{x\in\R^{n}:\inneri{y_1}{x} \in (w_{n,N},1]\} \cap B_n| .
\end{align}
By the definitions of $\varepsilon_{n,N}$ and $t_{n,N}$,  we have
\[ 
w_{n,N} >1-\frac{1}{2}\left(\frac{|\partial B_n|}{|B_{n-1}|N}\right)^{\frac{2}{n-1}}+\frac{c_1}{2n}\NN
=1-\frac{1}{2}\varepsilon_{n,N}^2+\frac{c_1}{2n}\NN.
\]
Hence, we can estimate the volume of each cap as
\begin{align}{\label{eq:a2}}
|\{x\in\R^{n}:\inneri{y_1}{x} \in (w_{n,N},1]\} \cap B_n|
&=|B_{n-1}|\int_{w_{n,N}}^{1}\left(1-x^{2}\right)^{\frac{n-1}{2}}dx \nonumber\\
& \leq 2^{\frac{n-1}{2}}|B_{n-1}|\int_{1-\frac{1}{2}\varepsilon_{n,N}^2+\frac{c_1}{2n}\NN}^{1}\left(1-x\right)^{\frac{n-1}{2}}dx\nonumber\\
& =\frac{|B_{n-1}|}{n+1}2^{\frac{n+1}{2}}\bigg(\frac{1}{2}\left(\frac{|\partial B_n|}{|B_{n-1}|N}\right)^{\frac{2}{n-1}}-\frac{c_1}{2n}\NN\bigg)^{\frac{n+1}{2}} \nonumber\\
& \leq\frac{|B_{n-1}|}{n+1}\bigg(\left(1-\frac{c_1}{n}\right)\left(\frac{|\partial B_n|}{|B_{n-1}|N}\right)^{\frac{2}{n-1}}\bigg)^{\frac{n+1}{2}} \nonumber\\
& < 9e^{-c_1/2}|B_n|N^{-1}N^{-\frac{2}{n-1}}.
\end{align}
Please note that in the last inequality, we used Lemma \ref{stirling} and the elementary inequality $\frac{n}{n+1}\sqrt{2\pi n}<9$, $n\geq 2$. Thus, from (\ref{eq:a1}) and (\ref{eq:a2}) we obtain
\[
|B_n \setminus \hat{P_b}| 
\leq 9e^{-c_1/2}|B_n|\NN.
\] 
However, by \cite[Theorem 2]{Lud06} it is known that
\[
|\hat{P_b} \triangle B_n| \geq c|B_n|\NN.
\]
Therefore, when $ c_1 $ is large enough we get
\[
|\hat{P_b}| \geq \left(1+\frac{c}{2}\NN\right)|B_n|,
\]
which contradicts the fact that $ |\hat{P_b}| = |B_n|. $ Hence, for all polytopes $P_{n,N}\supset B_n$ with $N$ facets, when $N$ is large enough we have $|P_{n,N}| > \left(1+\tfrac{1}{2}(n+\ln n- c_1)\NN\right)|B_n|$, i.e., 
\[
|P_{n,N}\setminus B_n| > \left(1+\frac{1}{2}(n+\ln n- c_1)\NN\right)|B_n|\NN. 
\]
Thus, from (\ref{GruberOutEqn}) we get that for all sufficiently large $N$,
\[
\frac{1}{2}\left(1-\frac{1}{n^2}\right)\divv_{n-1}|\partial B_n|^{\frac{n+1}{n-1}}\NN \geq \frac{1}{2}(n+\ln n-c_1)|B_n|\NN.
\]
Finally, another application of Stirling's inequality yields
\begin{align}\label{divgeq}
\divv_{n-1} \geq (2\pi e)^{-1}(n+\ln n) -c_1.
\end{align}
The theorem now follows from \eqref{divleq} and \eqref{divgeq}. \qed
\section{Proof of Theorem \ref{mahler}}\label{mahlerproof1}

First note that it is well-known that a centered convex body has the origin as its Santal\'o point (see, e.g., \cite{SchneiderBook}). Also, recall that $ P_{n,N} = \conv\{\pm X_i\}_{i=1}^{\frac{N}{2}}$. By an argument of M\"uller \cite{muller1990approximation}, it follows  that\footnote{The original proof in \cite{muller1990approximation} is for $\conv\{X_1,\ldots,X_{N}\}$; simple modifications to the arguments there show that the expected volume of $\conv\{\pm X_1,\ldots,\pm X_{N/2}\}$ equals the expected volume of $\conv\{ X_1,\ldots,X_{N}\}$ (up to a negligible factor).}
\[
\E[|P_{n,N}|] = \left(1-2^{-1}\left(n+4\ln n+O(1)\right)N^{-\frac{2}{n-1}}\right)|B_n|.
\]
Furthermore, Reitzner \cite{reitzner2003random} showed that when $ N $  is large enough,
\[
\var(|P_{n,N}|) \leq c(n)N^{-(1+\frac{4}{n-1})}.
\]
Thus, by Chebyshev's inequality the event
\begin{equation}\label{chevy1}
|P_{n,N}| \approx \left(1-2^{-1}(n+ 4\ln n+O(1))N^{-\frac{2}{n-1}}\right)|B_n|  
\end{equation}
holds with probability at least $1-C(n)N^{-1}$, where $\approx$ denotes asymptotic equality up to a factor of $O(1)\NN$. Now the proof of Theorem \ref{deldiv}, specifically the argument showing that $\divv_{n-1} \geq (2\pi e)^{-1}(n+\ln n) - c_1,$ implies that when $ N$ is large enough the following inequality holds for every realization of $ P_{n,N}$: 
\begin{equation}\label{polarlower1}
|P^{\circ}_{n,N}| \approx \left(1+2^{-1}(n+\ln n+O(1))\NN\right)|B_n| .
\end{equation}
Using (\ref{chevy1}) and (\ref{polarlower1}), we conclude that with high probability
\[
F(n,N) \geq \E[|P_{n,N}|\cdot|P^{\circ}_{n,N}|] \approx (1-(1.5\ln n + O(1))\NN)|B_n|^2.
\]
\qed

\section{Proof of Theorem  \ref{partialthm}}\label{spherethm} 

\subsection{Estimating the expectation of a random partial covering}\label{expectationsection}
Here we show that the expectation of the random variable $\mathbb{V}(\mathbf{X}_{\alpha N})$, which is the proportion of the surface area of $\Sn$ that the caps capture, is about $ 1-e^{-\alpha}.$ Let $\epsilon$ be the height of a spherical cap with a surface area $N^{-1}|\partial B_n|$. By Fubini's theorem and the independence of the $X_i$,
\begin{equation}\label{indepRV}
\begin{aligned} \E[\mathbb{V}(\mathbf{X}_{\alpha N})]=\E\left[\int_{\mathbb{S}^{n-1}}(1 -\mathbbm{1}_{\max\langle X_{i},x\rangle\leq 1-\epsilon}(x))\,d\sigma(x)\right]  &=1 -\int_{\mathbb{S}^{n-1}}\prod_{i=1}^{\lceil \alpha N \rceil}\E\big[\mathbbm{1}_{\langle X_{i},x\rangle\leq 1-\epsilon}(x)\big]d\sigma(x)\\
&= 1 - \left(1-\frac{1}{N}\right)^{\lceil \alpha N \rceil}
\\& =1 - e^{\lceil\alpha N\rceil  \ln(1-1/N)}
\\
&=1 - e^{-\alpha} + \frac{1}{2}\alpha e^{-\alpha} N^{-1} +O( \alpha e^{-\alpha} N^{-2}).
\end{aligned}
\end{equation}
Observe that each cap increases $ \mathbb{V}(\mathbf{X}_{\alpha N}) $ by at most $ N^{-1}$. Hence, for all $ 1\leq i \leq \lceil \alpha N \rceil $,
	\[
	c_i:=\sup_{ x_1,\ldots,x_{\lceil \alpha N \rceil},x'_i} | \mathbb{V}(\mathbf{X}_{\alpha N})- \mathbb{V}(\mathbf{X}_{\alpha N}^\prime)| \leq N^{-1}
	\]
	where $\mathbf{X}_{\alpha N}^\prime:=(X_1,\ldots,X_{i-1},X_i^\prime,X_{i+1},\ldots,X_{\lceil \alpha N \rceil})$. Thus, by McDiarmid's inequality, 
	\[
	\Pr(|\mathbb{V}(\mathbf{X}_{\alpha N}) - (1-e^{-\alpha})| > \epsilon) \leq 2\exp\left(-\frac{2\epsilon^2}{\sum_{i=1}^{\lceil \alpha N \rceil}N^{-2}}\right) \leq 2e^{-2\lfloor \alpha^{-1}\rfloor N\epsilon^2}.
	\]
	\subsection{Estimating the variance of a random partial covering}
	For this proof, we use the notation
	$
	    \epsilon=\epsilon(n,N) := \left(\frac{|\partial B_{n}|}{|B_{n-1}|N}\right)^{\frac{1}{n-1}},
	$
	and 
	$
	    C(x,N):= B_G(x,N^{-1})
	$
	denotes the spherical cap with center $x$ and normalized surface area $N^{-1}$. 
	We shall also use the fact that $  (1-\E[\mathbb{V}(\mathbf{X}_{\alpha N})])^2 = e^{-2\alpha}(1+0.5 \alpha N^{-1})^2 =  e^{-2\alpha}(1+\alpha N^{-1}) + O(C(\alpha)N^{-2}) $. 
	By elementary properties of the variance and the linearity of expectation,
	\begin{align}\label{variancecaps}
	    \var(\mathbb{V}(\mathbf{X}_{\alpha N})) 
	    &=\E[(1-\mathbb{V}(\mathbf{X}_{\alpha N}))^2]-(1-\E[\mathbb{V}(\mathbf{X}_{\alpha N})])^2.
	\end{align}
Expanding the product in the first term, we obtain
	\begin{align*}
	    (1-\mathbb{V}(\mathbf{X}_{\alpha N}))^2 &= \left(\int_{\Sn}\mathbbm{1}_{\max\langle X_{i},x\rangle\leq 1-\epsilon}(x)\,d\sigma(x)\right)^2
	    \\&=
	    \int_{\Sn}\int_{\Sn}\prod_{i=1}^{\lceil\alpha N\rceil}\mathbbm{1}_{\langle X_{i},x\rangle\leq 1-\epsilon,\langle X_{i},y\rangle\leq 1-\epsilon}(x)\,d\sigma(y)\,d\sigma(x).
	\end{align*}
By independence and the rotational invariance of the uniform measure, we get
	\begin{align*}
	    \E[ (1-\mathbb{V}(\mathbf{X}_{\alpha N}))^2] &= \int_{\Sn}\int_{\Sn}\Pr( x \notin C(X_i,N),y \notin C(X_i,N))^{\lceil\alpha N\rceil}\,d\sigma(y)\,d\sigma(x) 
	    \\&= \int_{\Sn}\Pr( e_n \notin C(X_i,N),y \notin C(X_i,N))^{\lceil\alpha N\rceil}\,d\sigma(y)
	    \\&= \int_{\Sn}\Pr( X \notin C(y,N), X \notin C(e_n,N))^{\lceil\alpha N\rceil}\,d\sigma(y).
   \end{align*}
The last integrand measures the probability that $e_n$ and $y$ are not inside the same cap centered at $X$. Since $X$ is drawn uniformly from the sphere, this probability equals measure of the union of these two spherical caps; moreover, if $d_G(y,e_n) \geq 2\epsilon$, then it equals the sum of the measures of the two caps. Thus,
    \begin{align}\label{2ndmoment}
        \E[(1-\mathbb{V}(\mathbf{X}_{\alpha N}))^2] &= (1-2^{n-1}N^{-1})(1-2N^{-1})^{\lceil\alpha N\rceil} \nonumber\\
&+ \int_{C(e_n,2^{n-1}N^{-1})}\Pr( X \notin C(y,N), X \notin C(e_n,N))^{\lceil\alpha N\rceil}\,d\sigma(y)  
       \nonumber \\
       & = (1+\alpha N^{-1}-2^{n-1}N^{-1})e^{-2\alpha} \nonumber\\
&+ \underbrace{\int_{C(e_n,2^{n-1}N^{-1})}\left(1 - \frac{| C(y,N) \cup C(e_n,N) |}{|\partial B_n|}\right)^{\lceil\alpha N\rceil}\,d\sigma(y)}_{=:I}.
    \end{align}
    Now when $N$ is large enough, we may assume that the caps are $(n-1)$-dimensional balls with the same radius. Indeed, if their distance is $t < 2\epsilon$, then the measure of the intersection of the two balls equals the measure of two $(n-1)$-balls with height $\epsilon - t/2$ (up to a negligible perturbation). Using polar coordinates, we  estimate the  integral $I$ by
    \begin{align}\label{integralI}
         I &=  \frac{|\partial B_{n-1}|}{|\partial B_n|}\int_{0}^{2\epsilon}t^{n-2}(1-2N^{-1}+\sigma(C(e_n,\epsilon-t/2)))^{\lceil\alpha N\rceil}\, dt
        \nonumber \\& \leq  \frac{|\partial B_{n-1}|}{|\partial B_n|}\int_{4^{-1}\epsilon}^{2\epsilon} t^{n-2}(1-(2+c^{n-1})N^{-1})^{\lceil\alpha N\rceil}\,dt + e^{-\alpha}4^{-(n-1)}N^{-1}
       \nonumber\\& \leq    2^{n-1}N^{-1}(1-(2+ c^{n-1})N^{-1})^{\lceil\alpha N\rceil}+e^{-\alpha}C_1^{n-1}N^{-1} 
       \nonumber\\&\leq 2^{n-1}N^{-1}e^{-2\alpha} + e^{-\alpha}C^{n-1}N^{-1}. 
    \end{align}
  Combining \eqref{2ndmoment} and \eqref{integralI}, we get 
    \begin{equation}\label{2ndmomentupper}
        \E[(1-\mathbb{V}(\mathbf{X}_{\alpha N}))^2] \leq e^{-2\alpha}(1+\alpha N^{-1}) + e^{-\alpha}C^{n-1}N^{-1}. 
    \end{equation}
    Putting everything together, we use \eqref{variancecaps} and \eqref{2ndmomentupper} to derive 
    \begin{align*}
\var(\mathbb{V}(\mathbf{X}_{\alpha N}))
&\leq e^{-2\alpha}(1+\alpha N^{-1}) + e^{-\alpha}C^{n-1}N^{-1}-\left(e^{-\alpha} - \frac{1}{2}\alpha e^{-\alpha} N^{-1} -O( \alpha e^{-\alpha} N^{-2})\right)^2\\
&\leq e^{-\alpha}C^{n-1}N^{-1}.
    \end{align*}
Following a similar analysis, one can also show that
    \[
         \E[(1- \mathbb{V}(\mathbf{X}_{\alpha N}))^2] \geq e^{-2\alpha}(1+\alpha N^{-1}) + c_2^{n-1} e^{-C_2\alpha} N^{-1},
    \]
    where $C_2 \in (1,2)$.
    The claim follows.

\section{Proof of Theorem \ref{geneThm}}\label{pfthm1}
In this section, we show how to modify the proof of Corollary \ref{mainThm}  to extend the result from the Euclidean ball to all smooth convex bodies $K$ with positive curvature. The proof that is given holds for the uniform distribution on the boundary; the extension to arbitrary densities follows from arguments similar to those in the proof of Corollary \ref{mainThm}. 

\vspace{1mm}

First, recall that the random polytope $ P_{n,N} $ is defined by
\[
P_{n,N} := \bigcap_{i=1}^{N}\{x\in\R^n: \inneri{x}{\nu(X_i)} \leq \inneri{X_i}{\nu(X_i)} \}, \hspace{.2in} X_1,\ldots,X_N \stackrel{\text{i.i.d.}}{\sim} \sigma_{\partial K} ,
\]
where $ \sigma_{\partial K} $ denotes the uniform probability measure on the boundary of $ K $. We use a ``polar'' coordinates formula for a convex body with the origin in its interior (see, e.g., \cite{nazarov2003maximal}) to express the volume of the set difference $P_{n,N}\setminus K$ as
\begin{align*}
|P_{n,N}\setminus K| &=\int_{\R^{n} \setminus K} \mathbbm{1}_{\max\inneri{\nu(X_i)}{x} \leq \inneri{\nu(X_i)}{X_i}}(x)\, dx\\
&= |\partial K|\int_{1}^{\infty}r^{n-1}\int_{\partial K}\|y\|\alpha(y)\mathbbm{1}_{\max\inneri{\nu(X_i)}{y}\leq r^{-1}\inneri{\nu(X_i)}{X_i}}(y) \,d\sigma_{\partial K}(y)\,dr,
\end{align*}
where  $ \nu(y) $ is the outer unit normal to $\partial K$ at the point $y$ and $ \alpha(y) $ is the cosine of  the angle between the normal $ \nu(y) $ and the ``radial" vector $ y\in\partial K.$ We now split this integral into three parts as we did in the proof of Corollary \ref{mainThm}:
\begin{align}|P_{n,N}\setminus K| & =|\partial K|\bigg(\int_{1}^{1+c_{K,N}}r^{n-1}\int_{\partial K}\|x\|\alpha(x)\mathbbm{1}_{\max\langle\nu(X_{i}),x\rangle\leq r^{-1}\langle\nu(X_{i}),X_{i}\rangle}(x)\,d\sigma_{\partial K}(x)\,dr\nonumber\\
& +\int_{1+c_{K,N}}^{1+\frac{2}{n}}r^{n-1}\int_{\partial K}\|x\|\alpha(x)\mathbbm{1}_{\max\langle\nu(X_{i}),x\rangle\leq r^{-1}\langle\nu(X_{i}),X_{i}\rangle}(x)\,d\sigma_{\partial K}(x)\,dr\nonumber\\
& +\int_{1+\frac{2}{n}}^{\infty}r^{n-1}\int_{\partial K}\|x\|\alpha(x)\mathbbm{1}_{\max\langle\nu(X_{i}),x\rangle\leq r^{-1}\langle\nu(X_{i}),X_{i}\rangle}(x)\,d\sigma_{\partial K}(x)\,dr\bigg)\nonumber\\
& =Y+Z+W.
\end{align}
Here $ c_{K,N}:=c(K)\NN(\ln N)^{\frac{2}{n-1}} $ and $ c(K) $ is a large constant that is defined at the end of Subsection \ref{sectionZnegligibleK}. As in the proof of Corollary \ref{mainThm}, we will divide the proof into three lemmas, considering each random variable $Y$, $Z$ and $W$ separately. The proofs of these lemmas are similar to those of Lemmas  \ref{Yprop}, \ref{Zprop} and \ref{Wconcentration} for the Euclidean unit ball. The modifications needed to extend the proofs to all smooth convex bodies  involve elementary differential geometry. 



\subsection{Concentration for the random variable $Y$}

\begin{lemma}\label{YpropK}
	Let $\epsilon>0$. There is a constant $c_1(K)>0$ such that for all sufficiently large $N$,
	\[
	\Pr(|Y-\E[Y]| \geq \epsilon) \leq
	2\exp\left(-c_1(K)N^{1+\frac{4}{n-1}}f(N)\epsilon^2\right)
	\]
	where $f(N)=(\ln N)^{-(2+\frac{4}{n-1})}$.
\end{lemma}

\begin{proof}
	
	We follow along the same lines as the proof of Lemma \ref{Yprop}, where we used McDiarmid's inequality to derive a concentration inequality for the random variable $Y$ by analyzing a random partial ``covering" of the sphere by geodesic balls of a fixed radius. However, in the setting of smooth convex bodies with positive curvature, we will instead consider a random partial ``covering" of $\partial K$ by geodesic ellipsoids. Moreover, unlike the sphere ``covering" setting, the  shape of each geodesic ellipsoid can vary depending on the curvature of $K$ at the ellipsoid's center.
	
	\vspace{1mm}
	
	First, we define $ f(X_1,\dots,X_N):=|P_{n,N}^c\cap(1+c_{K,N})K\cap K^c| $ to be the volume that we remove from $ (1+c_{K,N})K \setminus K $, so that $Y=|(1+c_{K,N})K\setminus K| - f(X_1,\ldots,X_N)$. Then for all $1\leq i\leq N$,
	\begin{align*}
	c_i:&=\sup_{x_1,\ldots,x_N,x'_i}|f(x_1,\ldots,x_i,\ldots ,x_N) - f(x_1,\ldots,x'_i,\ldots x_N)| \\
	&\leq \sup_{x_1,\ldots,x_N,x'_i} |(x_i^\perp)^{-}\cap (1+c_{K,N})K|.
	\end{align*}
	
	Since $K$ has $C^2$ boundary with  positive curvature, each point in the boundary of $K$ is an elliptic point. Thus for each $x\in\partial K$, we can represent the cap $(x_i^\perp)^{-} \cap  K$ in local coordinates as a cap of the ellipsoid  $\mathcal{E}(x)$ with axes length as the principal radii of curvature $\kappa_j(x)^{-1}$:
	\[
	\mathcal{E}(x):=\left\{\left(z_1,\ldots,z_{n-1},\sqrt{1 - \sum_{j=1}^{n-1}\frac{z^2_j}{\kappa_j(x)^{-2}}}\right): z_1,\ldots,z_{n-1}\in\R\right\}.
	\]Please note that when $N$ is large enough, for each $1\leq i\leq N$ the volume of the cap $(x_i^\perp)^{-} \cap  K$ equals the volume of  $(x_i^\perp)^{-}\cap \mathcal{E}(x_i)$, up to a term of negligible order in $N$. 
	Using a computation similar to the one in the proof of Corollary \ref{mainThm}, we derive that
	\begin{align}\label{ellipsoidcap}
	|(x_i^\perp)^{-}\cap (1+c_{K,N})\mathcal{E}(x_i)| &=|B_{n-1}|\int_{\langle\nu(x_i),x_i\rangle}^{1+c_{K,N}}\left((1+c_{K,N})^2-x^2\right)^{\frac{n-1}{2}}\prod_{j=1}^{n-1}\kappa_j(x)^{-1}\,dx \nonumber\\
	&\leq 3^{\frac{n-1}{2}}(1+c_{K,N})^{\frac{n-1}{2}}|B_{n-1}|\kappa(x_i)^{-1}\int_{\langle\nu(x_i),x_i\rangle}^{1+c_{K,N}}\left(1+c_{K,N}-x\right)^{\frac{n-1}{2}}\,dx \nonumber\\
	&\leq C(K)N^{-(1+\frac{2}{n-1})}(\ln N)^{1+\frac{2}{n-1}}, \nonumber
	\end{align}where $\kappa(x)$ denotes the Gaussian curvature of $K$ at $x\in\partial K$. The rest of the proof is similar to that of Lemma \ref{Yprop}. 
\end{proof}


\subsection{The random variable $Z$ is negligible}\label{sectionZnegligibleK}

Next, we extend Proposition \ref{Zprop} from the ball to all smooth convex bodies with positive curvature. The final ingredient we need follows from the papers \cite{erdHos1964amount,erdos1953covering} of Erd{\H{o}}s and Rogers. 

\begin{lemma}{\label{ErodsRogeres}}
	Let $ K $ be a $ C^2 $ convex body, and for fixed $\delta>0$ let $ \mathcal{N} $ be  a minimal $\delta$-net of $\partial K$, i.e., $\bigcup_{x\in\mathcal N}B_G(x,\delta) \supset \partial K$ and every $\delta$-net of $\partial K$ contains at least $|\mathcal N|$ elements. 
	Then when $ \delta $ is small enough, each point of $\partial K$ lies in the interior of no more than $ 4^n\frac{|\partial K|}{|\partial B_n|} $ balls of the covering.
\end{lemma}
This result is far from optimal; recall that B\"or\"oczky and Wintsche \cite{boroczky2003covering} showed that for a Euclidean ball of any radius, there is a covering such that each point of its boundary lies in the interior of no more than $400n\ln n$ geodesic balls.

\vskip .1in

The extension of Lemma \ref{Zprop} is given in the next lemma.
\begin{lemma}\label{ZpropK}
	\begin{equation*}	
	\Pr\left(Z \geq N^{-(0.5+\frac{2}{n-1})}\right) \leq \exp\left(-c_1(K)N^{0.5-\frac{2}{n-1}}\right).
	\end{equation*}
\end{lemma}

\begin{proof}
	Without loss of generality, we can assume that the origin lies in the interior of $K$, so that $\max_{x\in \partial K}\|x\| \leq \diam(K)$. 
	Moreover, $\alpha(x) := \cos(\measuredangle(\nu(x),x)) \leq 1$ for any $x\in\partial K$. Thus,
	\begin{align*}Z & =|\partial K|\int_{1+c_{K,N}}^{1+\frac{2}{n}}r^{n-1}\int_{\partial K}\|x\|\alpha(x)\mathbbm{1}_{\max\langle\nu(X_{i}),x\rangle\leq r^{-1}\langle\nu(X_{i}),X_{i}\rangle}(x)\,d\sigma_{\partial K}(x)\,dr\\
	& \leq\diam(K)|\partial K|\int_{1+c_{K,N}}^{1+\frac{2}{n}}r^{n-1}\,dr\int_{\partial K}\mathbbm{1}_{\max\langle\nu(X_{i}),x\rangle\leq(1+c_{K,N})^{-1}\langle\nu(X_{i}),X_{i}\rangle}(x)\,d\sigma_{\partial K}(x)\\
	& \leq Cn^{-1}\diam(K)|\partial K|\int_{\partial K}\mathbbm{1}_{\max\langle\nu(X_{i}),x\rangle\leq(1+c_{K,N})^{-1}\langle\nu(X_{i}),X_{i}\rangle}(x)\,d\sigma_{\partial K}(x).
	\end{align*}
As in the proof of Lemma \ref{Zprop}, we  define 
\[
\tilde Z := |\partial K|\int_{\partial K} \mathbbm{1}_{\max\langle\nu(X_{i}),x\rangle\leq(1+c_{K,N})^{-1}\langle\nu(X_{i}),X_{i}\rangle}(x)\,d\sigma_{\partial K}(x).
\] 
The random variable $\tilde Z$  measures the missing surface area of a random  ``covering" of $K$ by $ N $ random geodesic ellipsoids of volume $ c(K)\frac{\ln N}{N}$. By independence, its expected value  can be estimated by
	\begin{align*}\E[\tilde{Z}] & = |\partial K| \prod_{i=1}^{N}\Pr\left(\langle\nu(X_{i}),x\rangle\leq(1+c_{K,N})^{-1}\langle\nu(X_{i}),X_{i}\rangle\right)\\
	& \leq |\partial K|\left(1-\frac{c(K)^{n-1}C(K)^{\frac{n-1}{2}}|B_{n-1}|(\ln N)}{N|\partial K|}\right)^{N}\\
	& \leq |\partial K|\exp\left(-\frac{c(K)^{n-1}C(K)^{\frac{n-1}{2}}|B_{n-1}|}{|\partial K|}\cdot\ln N\right)\\
	& =|\partial K| N^{-C(K)^{\frac{n-1}{2}}c(K)^{n-1}|B_{n-1}|/|\partial K|}.
	\end{align*}
	
	Next, we show that $ \tilde{Z} $ is negligible with high probability. As in the proof of Lemma \ref{Zprop}, we can reduce the problem to the following random ``covering" of $K$ by $C(K)N\ln N$ random ellipsoids of volume $c_i(K)N^{-1}$, $1\leq i\leq C(K)N\ln N$. Let $\mathbf{X} := \{X_1,\ldots, X_{C(K)N\ln N}\} $ and define 
	\[
	A_{\mathbf{X}} :=\bigcup_{i=1}^{C(K)N\ln N} \mathcal{E}(X_i,c_i(K)N^{-1}) ,
	\]
	where $ \mathcal{E}(X_i,c_i(K)N^{-1})  \subset \partial K$ denotes the geodesic ellipsoid centered at $X_i$ with volume $c_i(K)N^{-1}$.
	In particular, every ellipsoid contains a ball of radius 
	\[
	r_{N,K}:= c_1(K)\min_{x\in\partial K}\min_{1\leq i\leq n-1}\kappa_i(x)^{-1}N^{-\frac{1}{n-1}}.
	\]
	Now as in Lemma \ref{Zprop}, it suffices to prove that
	\[
	\Pr\left(\tilde Z \geq N^{-(0.5+\frac{2}{n-1})}\right) \leq \exp\left(-c_2(K)N^{-(0.5+\frac{2}{n-1})}\right) .
	\]
	In order to apply the same proof of Lemma \ref{Zprop},  we need  a covering of $\partial K$ by geodesic balls of radius $ \frac{1}{2}r_{N,K} $ such that each point is counted no more than $ c_4(K)  $ times. Indeed, Lemma \ref{ErodsRogeres} provides such a covering. Now the rest of the proof proceeds similarly to that of Lemma \ref{Zprop}, and we ultimately derive that
	\[
	\Pr(B) \leq  \exp\left(c_2(K)(\ln N)N^{0.5-\frac{2}{n-1}}-c(K)c_1(K)(\ln N)N^{0.5-\frac{2}{n-1}}\right).
	\]
	Choosing $c(K)$ to be large enough yields the lemma.
\end{proof}



Finally, we turn our attention to the random variable $W$. 

\begin{lemma}\label{WpropK}
	When $N$ is large enough, the polytope $P_{n,N}$ lies in $(1+2/n)K$ with probability at least $1-e^{-c(K,\mu)N}$.
\end{lemma}
The proof is similar to that of Lemma~ \ref{Wconcentration}, where now we replace the spherical caps by geodesic ellipsoids. We leave the details to the interested reader. 
\subsection*{Conclusion of the proof of Theorem \ref{geneThm}.}

Recalling that $|P_{n,N} \setminus K| = Y + Z+W$,  the rest of the proof of Theorem \ref{geneThm} proceeds in the same way as the proof of Corollary \ref{mainThm} in Subsection \ref{proofsectionmainthm}. \qed


\bibliographystyle{spmpsci}      
\bibliography{biblo}   

\begin{thebibliography}{10}
\providecommand{\url}[1]{{#1}}
\providecommand{\urlprefix}{URL }
\expandafter\ifx\csname urlstyle\endcsname\relax
  \providecommand{\doi}[1]{DOI~\discretionary{}{}{}#1}\else
  \providecommand{\doi}{DOI~\discretionary{}{}{}\begingroup
  \urlstyle{rm}\Url}\fi

\bibitem{alexander2017polytopes}
Alexander, M., Fradelizi, M., Zvavitch, A.: Polytopes of {M}aximal {V}olume
  {P}roduct.
\newblock Discrete and Computational Geometry \textbf{62}(3), 583--600 (2019)

\bibitem{artin}
Artin, E.: The Gamma Function.
\newblock Dover, Mineola, New York (2015)

\bibitem{aubrun2017alice}
Aubrun, G., Szarek, S.J.: Alice and Bob Meet Banach: The Interface of
  Asymptotic Geometric Analysis and Quantum Information Theory,
  \emph{Mathematical Surveys and Monographs}, vol. 223.
\newblock American Mathematical Society (2017)

\bibitem{Blaschke1917}
Blaschke, W.: {\"U}ber affine {G}eometrie {III}: {E}ine {M}inimumeigenschaft
  der {E}llipse.
\newblock Berichte \"uber die {V}erhandlungen der k\"onigl. s\"achs
  {G}esellschaft der {W}issenschaften zu {L}eipzig, Math. Nat. Klasse
  \textbf{69}, 3--12 (1917)

\bibitem{Blaschke1918}
Blaschke, W.: Affine geometrie xiv: Eine minimumaufgabe f\"ur legendres
  tr\"agheitsellipsoid.
\newblock Ber. Verh. S\"achs. Akad. Wiss. Leipzig Math.-Phys. Kl. \textbf{70},
  72--75 (1918)

\bibitem{boroczky2000polytopal}
B{\"o}r{\"o}czky, K.: Polytopal approximation bounding the number of $k$-faces.
\newblock Journal of Approximation Theory \textbf{102}(2), 263--285 (2000)

\bibitem{boroczky2004finite}
B{\"o}r{\"o}czky, K.: Finite {P}acking and {C}overing, \emph{Cambridge Tracts
  in Mathematics}, vol. 154.
\newblock Cambridge University Press (2004)

\bibitem{boroczky2004approximation}
B{\"o}r{\"o}czky, K., Reitzner, M.: Approximation of smooth convex bodies by
  random circumscribed polytopes.
\newblock The Annals of Applied Probability \textbf{14}(1), 239--273 (2004)

\bibitem{boroczky2003covering}
B{\"o}r{\"o}czky, K., Wintsche, G.: Covering the {S}phere by {E}qual
  {S}pherical {B}alls.
\newblock In: B.~Aronov, S.~Basu, J.~Pach, M.~Sharir (eds.) Discrete and
  Computational Geometry: The Goodman-Pollack Festschrift, pp. 235--251.
  Springer (2003)

\bibitem{chen}
Chen, L.: New analysis of the sphere covering problems and optimal polytope
  approximation of convex bodies.
\newblock Journal of Approximation Theory \textbf{133}, 134--145 (2005)

\bibitem{diakonikolas2016learning}
Diakonikolas, I., Kane, D.M., Stewart, A.: Learning multivariate log-concave
  distributions pp. 711--727 (2017)

\bibitem{edelsbrunner}
Edelsbrunner, H.: Geometric algorithms.
\newblock In: Handbook of Convex Geometry, pp. 699--735. Elsevier,
  North-Holland (1993)

\bibitem{edelsbrunnermesh}
Edelsbrunner, H.: Geometry and {T}opology for {M}esh {G}eneration.
\newblock In: Cambridge Monographs on Applied and Computational Mathematics.
  Cambridge University Press, Cambridge (2001)

\bibitem{erdHos1964amount}
Erd{\H{o}}s, P., Few, L., Rogers, C.A.: The amount of overlapping in partial
  coverings of space by equal spheres.
\newblock Mathematika \textbf{11}(2), 171--184 (1964)

\bibitem{erdos1953covering}
Erd{\H{o}}s, P., Rogers, C.A.: The {C}overing of $n$-{D}imensional {S}pace by
  {S}pheres.
\newblock Journal of the London Mathematical Society \textbf{1}(3), 287--293
  (1953)

\bibitem{fodor2016volume}
Fodor, F., Hug, D., Ziebarth, I.: The volume of random polytopes circumscribed
  around a convex body.
\newblock Mathematika \textbf{62}(1), 283--306 (2016)

\bibitem{equipartition}
Fradelizi, M., Hubard, A., Meyer, M., Rold\'an-Pensado, E., Zvavitch, A.:
  Equipartitions and {M}ahler volumes of symmetric convex bodies.
\newblock arXiV preprint  (2019)

\bibitem{gardner1995geometric}
Gardner, R.J.: Geometric {T}omography, 2 edn.
\newblock Cambridge University Press (2006)

\bibitem{GKM}
Gardner, R.J., Kiderlen, M., Milanfar, P.: Convergence of algorithms for
  reconstructing convex bodies and directional measures.
\newblock The Annals of Statistics \textbf{34}, 1331--1374 (2006)

\bibitem{glasauer1996asymptotic}
Glasauer, S., Schneider, R.: Asymptotic approximation of smooth convex bodies
  by polytopes.
\newblock In: Forum Mathematicum, vol.~8, pp. 363--378. Walter de Gruyter,
  Berlin/New York (1996)

\bibitem{GruberIn}
Gruber, P.M.: Volume approximation of convex bodies by inscribed polytopes.
\newblock Mathematische Annalen \textbf{281}(2), 229--245 (1988)

\bibitem{GruberOut}
Gruber, P.M.: Volume approximation of convex bodies by circumscribed polytopes.
\newblock In: Applied Geometry and Discrete Mathematics: The Victor Klee
  Festschrift, vol.~4, pp. 309--317. American Mathematical Society (1991)

\bibitem{GruberII}
Gruber, P.M.: Asymptotic estimates for best and stepwise approximation of
  convex bodies {II}.
\newblock Forum Math. \textbf{5}, 521--538 (1993)

\bibitem{Wer15}
Hoehner, S.D., Sch{\"u}tt, C., Werner, E.M.: The {S}urface {A}rea {D}eviation
  of the {E}uclidean {B}all and a {P}olytope.
\newblock Journal of Theoretical Probability \textbf{31}(1), 244--267 (2018)

\bibitem{iriyehshibata}
Iriyeh, H., Shibata, M.: Symmetric {M}ahler's conjecture for the volume product
  in the three dimensional case.
\newblock arXiv:1706.01749  (2017)

\bibitem{janson1986random}
Janson, S.: Random coverings in several dimensions.
\newblock Acta Mathematica \textbf{156}(1), 83--118 (1986)

\bibitem{kufer1994approximation}
K{\"u}fer, K.H.: On the approximation of a ball by random polytopes.
\newblock Advances in Applied Probability \textbf{26}(4), 876--892 (1994)

\bibitem{kur2017approximation}
Kur, G.: Approximation of the {E}uclidean ball by polytopes with a restricted
  number of facets.
\newblock to appear in Studia Math., arXiv:1705.00210v6  (2019)

\bibitem{ludwig1999asymptotic}
Ludwig, M.: Asymptotic approximation of smooth convex bodies by general
  polytopes.
\newblock Mathematika \textbf{46}(01), 103--125 (1999)

\bibitem{Lud06}
Ludwig, M., Sch{\"u}tt, C., Werner, E.: Approximation of the {E}uclidean ball
  by polytopes.
\newblock Studia Mathematica \textbf{173}(1), 1--18 (2006)

\bibitem{mahler2}
Mahler, K.: Ein minimalproblem f\"ur konvexe polygone.
\newblock Mathematica (Zutphen) B \textbf{7}, 118--127 (1939)

\bibitem{mahler1}
Mahler, K.: Ein ubertragungsprinzip f\"ur konvexe k\"orper.
\newblock \v{C}asopis P\v{e}st. Mat. Fys \textbf{68}, 93--102 (1939)

\bibitem{MaS1}
Mankiewicz, P., Sch\"utt, C.: A simple proof of an estimate for the
  approximation of the {E}uclidean ball and the {D}elone triangulation numbers.
\newblock Journal of Approximation Theory \textbf{107}, 268--280 (2000)

\bibitem{MaS2}
Mankiewicz, P., Sch\"utt, C.: On the {D}elone triangulation numbers.
\newblock Journal of Approximation Theory \textbf{111}, 139--142 (2001)

\bibitem{McVi}
McClure, D.E., Vitale, R.A.: Polygonal approximation of plane convex bodies.
\newblock Journal of Mathematical Analysis and Applications \textbf{51},
  326--358 (1975)

\bibitem{McDiarmid1989}
McDiarmid, C.: On the method of bounded differences.
\newblock In: Surveys in Combinatorics 1989, pp. 148--188. Cambridge University
  Press (1989)

\bibitem{meyer1991}
Meyer, M.: Convex bodies with minimal volume product in $\mathbb{R}^2$.
\newblock Monatsh. Math. \textbf{112}, 297--301 (1991)

\bibitem{muller1990approximation}
M{\"u}ller, J.S.: Approximation of a ball by random polytopes.
\newblock Journal of Approximation Theory \textbf{63}(2), 198--209 (1990)

\bibitem{nazarov2003maximal}
Nazarov, F.: On the {M}aximal {P}erimeter of a {C}onvex {S}et in $\mathbb{R}^n$
  with {R}espect to a {G}aussian {M}easure.
\newblock In: V.D. Milman, G.~Schechtman (eds.) Geometric Aspects of Functional
  Analysis (Israel Seminar 2001--2002), \emph{Lecture Notes in Mathematics},
  vol. 1807, pp. 169--187. Springer (2003)

\bibitem{NPRZ}
Nazarov, F., Petrov, F., Ryabogin, D., Zvavitch, A.: A remark on the {M}ahler
  conjecture: Local minimality of the unit cube.
\newblock Duke Mathematical Journal \textbf{154}(3), 419--430 (2010)

\bibitem{petty1985}
Petty, C.M.: Affine isoperimetric problems.
\newblock Ann. N. Y. Acad. Sc. \textbf{440}, 113--127 (1985)

\bibitem{reitzner2003random}
Reitzner, M.: Random polytopes and the {E}fron--{S}tein jackknife inequality.
\newblock The Annals of Probability \textbf{31}(4), 2136--2166 (2003)

\bibitem{reitzner2005central}
Reitzner, M.: Central limit theorems for random polytopes.
\newblock Probability Theory and Related Fields \textbf{133}(4), 483--507
  (2005)

\bibitem{rogers1957note}
Rogers, C.A.: A note on coverings.
\newblock Mathematika \textbf{4}(1), 1--6 (1957)

\bibitem{rogers1963covering}
Rogers, C.A.: Covering a sphere with spheres.
\newblock Mathematika \textbf{10}(2), 157--164 (1963)

\bibitem{saintraymond1981}
Saint-Raymond, J.: Sur le volume des corps convexes sym\'etriques.
\newblock S\'eminaire Choquet - Initiation \'a l’Analyse 1980/81 Exp. No. 11
  pp. 1--25 (1981)

\bibitem{santalo1949}
Santal\'o, L.A.: Un invariante afin para los cuerpos convexos del espacio de
  $n$ dimensiones.
\newblock Portugal. Math. \textbf{8}, 155--161 (1949)

\bibitem{SchneiderBook}
Schneider, R.: Convex {B}odies: {T}he {B}runn-{M}inkowski {T}heory, 2 edn.
\newblock Cambridge University Press (2014)

\bibitem{schutt2003polytopes}
Sch{\"u}tt, C., Werner, E.: Polytopes with vertices chosen randomly from the
  boundary of a convex body.
\newblock In: V.D. Milman, G.~Schechtman (eds.) Geometric Aspects of Functional
  Analysis (Israel Seminar 2001--2002), \emph{Lecture Notes in Mathematics},
  vol. 1807, pp. 241--422. Springer (2003)

\bibitem{toth}
T\'oth, L.F.: Lagerungen in der {E}bene auf der {K}ugel und im {R}aum,
  \emph{Grundlehren der mathematischen Wissenschaften}, vol.~65, 2 edn.
\newblock Springer-Verlag (1972)

\bibitem{vu2005sharp}
Vu, V.: Sharp concentration of random polytopes.
\newblock Geometric \& Functional Analysis GAFA \textbf{15}(6), 1284--1318
  (2005)

\bibitem{vu2006central}
Vu, V.: Central limit theorems for random polytopes in a smooth convex set.
\newblock Advances in Mathematics \textbf{207}(1), 221--243 (2006)

\bibitem{zador1982asymptotic}
Zador, P.: Asymptotic quantization error of continuous signals and the
  quantization dimension.
\newblock IEEE Transactions on Information Theory \textbf{28}(2), 139--149
  (1982)

\end{thebibliography}


\end{document}